\documentclass[12pt]{amsart}
\usepackage{mathrsfs}
\usepackage{amsmath}
\usepackage{amsfonts}
\usepackage{amssymb}
\usepackage[all]{xy}           
\usepackage{bm}
\usepackage{bbding}
\usepackage{txfonts}
\usepackage{amscd}
\usepackage{xspace}
\usepackage[shortlabels]{enumitem}
\usepackage{ifpdf}

\ifpdf
  \usepackage[colorlinks,final,backref=page,hyperindex]{hyperref}
\else
  \usepackage[colorlinks,final,backref=page,hyperindex,hypertex]{hyperref}
\fi
\usepackage{tikz}
\usepackage[active]{srcltx}

\makeatletter

\newtheorem{thm}{Theorem}[section]
\newtheorem{lem}[thm]{Lemma}
\newtheorem{cor}[thm]{Corollary}
\newtheorem{pro}[thm]{Proposition}
\newtheorem{ex}[thm]{Example}
\theoremstyle{definition}
\newtheorem{rmk}[thm]{Remark}
\newtheorem{defi}[thm]{Definition}

\newcommand{\nc}{\newcommand}
\newcommand{\delete}[1]{}

\nc{\mlabel}[1]{\label{#1}}  
\nc{\mcite}[1]{\cite{#1}}  
\nc{\mref}[1]{\ref{#1}}  
\nc{\mbibitem}[1]{\bibitem{#1}} 

\delete{
\nc{\mlabel}[1]{\label{#1}{\hfill \hspace{1cm}{\bf{{\ }\hfill(#1)}}}}
\nc{\mcite}[1]{\cite{#1}{{\em{{\ }(#1)}}}}  
\nc{\mref}[1]{\ref{#1}{{\em{{\ }(#1)}}}}  
\nc{\mbibitem}[1]{\bibitem[\em #1]{#1}} 
}

\newcommand {\emptycomment}[1]{}

\nc{\oprn}{\theta}
\nc{\Oprn}{\Theta}

\nc{\calo}{\mathcal{O}}
\nc{\oop}{$\mathcal{O}$-operator\xspace}
\nc{\oops}{$\mathcal{O}$-operators\xspace}
\nc{\mrho}{{\bm{\varrho}}}
\nc{\emk}{\mathbf{K}}
\nc{\invlim}{\displaystyle{\lim_{\longleftarrow}}\,}
\nc{\ot}{\otimes}

\newcommand{\lon }{\,\rightarrow\,}
\newcommand{\be }{\begin{equation}}
\newcommand{\ee }{\end{equation}}

\newcommand{\g}{\mathfrak g}
\newcommand{\h}{\mathfrak h}

\newcommand{\huaL}{\mathcal{L}}
\newcommand{\huaR}{\mathcal{R}}

\newcommand{\huaG}{\mathcal{G}}

\newcommand{\huaX}{\mathcal{X}}
\newcommand{\huaY}{\mathcal{Y}}

\newcommand{\huaD}{\mathcal{D}}

\newcommand{\huaO}{{\mathcal{O}}}

\newcommand{\frkg}{\mathfrak g}

\newcommand{\frkL}{\mathfrak L}

\newcommand{\frkR}{\mathfrak R}

\newcommand{\half}{\frac{1}{2}}

\newcommand{\Courant}[1]{\left\llbracket  #1\right\rrbracket }

\newcommand{\Id}{{\rm{Id}}}

\newcommand{\br}[1]{   [ \cdot,    \cdot  ]   }

\newcommand{\Hom}{\mathrm{Hom}}

\newcommand{\gl}{\mathfrak {gl}}

\newcommand{\de}{\mathrm{d}}

\nc{\CV}{\mathbf{C}}

\begin{document}

\title{Twisting of Lie triple systems, $L_\infty$-algebras, and (generalized) matched pairs}

\author{Jia Zhao}
\address{Jia Zhao, School of Mathematics and Statistics, Nantong University, Nantng 226019, Jiangsu, China}
\email{zhaojia@ntu.edu.cn}

\author{Haobo Xia*}
\address{Haobo Xia (corresponding author), Department of Mathematics, Jilin University, Changchun 130012, Jilin, China}
\email{xiahb21@mails.jlu.edu.cn}

\date{\today}

\begin{abstract}
In this paper, we introduce notions of  (proto-, quasi-)twilled Lie triple systems and give their equivalent descriptions using the controlling algebra and bidegree convention. Then we construct an $L_\infty$-algebra via a twilled Lie triple system. Besides, we establish the twisting theory of Lie triple systems and then characterize the twisting as a Maurer-Cartan element in the constructed $L_\infty$-algebra. Finally, we clarify the relationship between twilled  Lie triple systems and matched pairs and clarify the relationship between twilled Lie triple systems and relative Rota-Baxter operators respectively so that we obtain the relationship between matched pairs of Lie triple systems and relative Rota-Baxter operators.
\end{abstract}


\keywords{twilled Lie triple system, $L_\infty$-algebra, twisting, matched pair, realtive Rota-Baxter operator}

\maketitle

\vspace{-1.1cm}

\tableofcontents

\allowdisplaybreaks

 \section{Introduction}
Lie triple systems can be traced back to Cartan's study on symmetric spaces \cite{Cartan}. It was abstracted as an algebraic object and was named as Lie triple systems by Jacobson \cite{Jacobson}, and then representation of Lie triple systems was defined by Hodge and Parshall \cite{Hodge}. Indeed, Leibniz algebras and Nambu algebras are related with Lie triple systems closely and Lie triple systems play an important role in Lie theory. The relationship between Lie triple systems and other algebraic structures is reflected in the following two aspects: on the one hand, a Lie triple system is a special Nambu algebra; on the other hand, the space of fundamental objects of Lie triple systems is endowed with a Leibniz algebraic structure. Due to the importance of Lie triple systems, Lister established a structure theory of Lie triple systems \cite{Lister}. Moreover, applications of Lie triple systems on numerical analysis of differential equations were studied in \cite{Kaas} and that $T^*$-extension of Lie triple systems is compatible with nilpotency and solvability was examined in \cite{Deng}.

Motivated by studies on (quasi-)Lie bialgebras and (quasi-)Hopf algebras, Drinfeld introduced an operation called the twisting in \cite{Drinfeld}. One provides a method to study Manin triples via the twisting operations. Twisting operations play an important role in the context of bialgebra theory and Poisson geometry, see \cite{Kosmann1,Kosmann2,Roytenberg1,Roytenberg2} for more details. It is well known that a graded commutative algebra $\wedge^\bullet (V\oplus V^*)$ is equipped with a Poisson bracket $\{\cdot,\cdot\}$ defined to be $\{v,v'\}=\{\epsilon,\epsilon'\}=0$ and $\{v,\epsilon\}=\langle v,\epsilon\rangle$ for any $v,v'\in V$ and $\epsilon,\epsilon'\in V^*$. Then a Lie algebra structure on $V\oplus V^*$ is an element $\Theta$ in $\wedge^3 (V\oplus V^*)$ such that $\{\Theta,\Theta\}=0$. Moreover, the structure $\Theta$ has a close connection with Lie bialgebra structures. A Lie bialgebra structure is  pair $(\mu,\nu)$, such that $\Theta:=\mu+\nu$ is a Lie algebra structure on $V\oplus V^*$, where $\mu\in (\wedge^2V^*)\ot V$ and $\nu\in V^*\ot(\wedge^2V)$. If $(\mu,\nu)$ forms a Lie bialgebra structure on $V$, then $(V\oplus V^*,\mu+\nu)$ is called a Drinfeld double. Suppose that $r\in \wedge^2V$, then by definition, the twisting of a structure $\Theta$ by $r$ is a transformation
$$\Theta^r:=\exp(X_r)(\Theta),$$
where $X_r$ is a Hamiltonian vector field $X_r:=\{\cdot,r\}$. The relationship between twisting and the classical Yang-Baxter equation is reinforced in the sequel. Let $(\mu,\nu=0)$ be a Lie bialgebra on the vector space $V$, then the Drinfeld double is the space $V\oplus V^*$ with the structure $\Theta:=\mu$. If $r$ is a solution to the classical Yang-Baxter equation
$$[r,r]=0,$$
then a pair $(\mu,\{\mu,r\})$ is a Lie bialgebra structure and the double $\mu+\{\mu,r\}=\Theta^r$, where $[r,r]:=\{\{\mu,r\},r\}$. Generally, Lie 2-(co-, bi-)algebra structures are characterized as Maurer-Cartan elements in a graded Poisson algebra $(S^\bullet(V[2]\oplus V^*[1]),\{\cdot,\cdot\})$ in \cite{CSX}.

Thanks to our study on Lie-Yamaguti bialgebra theory in \cite{ZQ}, we obtain the bialgebra theory of Lie triple systems naturally. In order to establish a comprehensive theoretical system for Lie triple system bialgebra, the aim of this paper is to build the twisting theoty of Lie triple systems. We introduce the notions of (proto-, quasi-)twilled Lie triple systems and then we construct an $L_\infty$-algebra  from a twilled Lie triple system. Consequently, we characterize the twisting as a Maurer-Cartan element of the constructed $L_\infty$-algebra. Finally, we introduce notions of matched pairs and generalized matched pairs of Lie triple systems to clarify the relationship between (generalized) matched pairs and relative Rota-Baxter operators. Similar to the case of $3$-Lie algebras, the reason why a solution to the classical Yang-Baxter equation (treated as a special relative Rota-Baxter operator) can {\em not} give rise to a double Lie triple system bialgebra is that a relative Rota-Baxter operator does {\em not} corresponds to a usual matched pair of Lie triple systems, but corresponds to a {\em generalized matched pair}. Relations among twilled Lie triple systems (twilled LTSs for short in the diagram), relative Rota-Baxter operators (relative RB-operators for short in the diagram) and matched pairs of Lie triple systems can be shown in the following diagram:
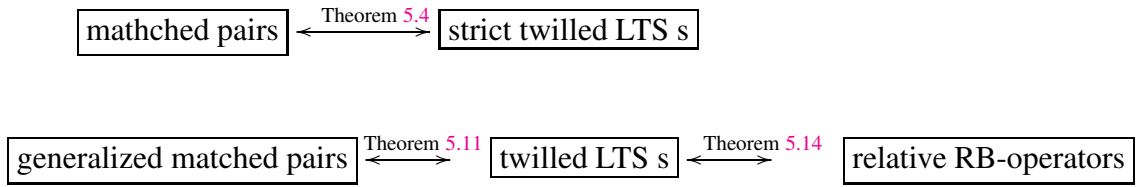
\[
\xymatrix{
\fbox{mathched pairs} \ar[r]^{\text{Theorem \ref{strict}}} &\fbox{strict twilled LTS s}\ar[l] & \\
\fbox{generalized matched pairs} \ar[r]^{\quad\qquad\text{Theorem \ref{usual}}} & \quad\fbox{twilled LTS s} \ar[l] \ar[r]^{\text{Theorem \ref{Otwisting}}} & \qquad\fbox{relative RB-operators} \ar[l]
}
\]
The conclusions shown in the above diagram indicate that generalized matched pairs of Lie triple systems may lead to generalized bialgebra theory for Lie triple systems and we expect new studies in this direction. Besides, one can read \cite{HST,T.S2,Uh1} for more details about twisting theory of associative algebras, Leibniz algebras and $3$-Lie algebras, and see \cite{BCM2,B.G.S} for bialgebra theories for left-symmetric algebras and 3-Lie algebras.

 In this paper, all the vector spaces are over $\mathbb{K}$, a field of characteristic $0$.

\section{Preliminaries: Lie triple systems and their controlling algebras}
In this section, we first recall some basic notions such as Lie triple systems, representations and their controlling algebras.

\begin{defi}\cite{Jacobson}\label{LY}
A {\bf Lie triple system} is a vector space $\g$, together with a trilinear bracket $\Courant{\cdot,\cdot,\cdot}:\wedge^2\g \otimes  \mathfrak{g} \to \mathfrak{g} $ such that the following equations are satisfied for all $x,y,z,w,t \in \g$,
\begin{eqnarray}
~ &&\label{LY1}\Courant{x,y,z}+\Courant{y,z,x}+\Courant{z,x,y}=0,\\
~ &&\Courant{x,y,\Courant{z,w,t}}=\Courant{\Courant{x,y,z},w,t}+\Courant{z,\Courant{x,y,w},t}+\Courant{z,w,\Courant{x,y,t}}.\label{fundamental}
\end{eqnarray}
Here, Eq. \eqref{fundamental} is called the {\bf fundamental identity}. We denote a Lie triple system by a pair $(\g,\Courant{\cdot,\cdot,\cdot})$.
\end{defi}

\begin{rmk}
A Lie triple system structure induces a Leibniz algebra structure on its tensor space. Let $(\g,\Courant{\cdot,\cdot,\cdot})$ be a Lie triple system. Define an operation $[\cdot,\cdot]_{\mathsf F}$ on $\ot^2\g$ to be
$$[\huaX,\huaY]_{\mathsf F}:=\Courant{x_1,x_2,y_1}\ot y_2+y_1\ot\Courant{x_1,x_2,y_2},\quad \forall \huaX=x_1\ot x_2,~\huaY=y_1\ot y_2\in \ot^2\g,$$
then $(\ot^2\g,[\cdot,\cdot]_{\mathsf F})$ is a Leibniz algebra. Elements in $\ot^2\g$ are called the {\bf fundamental objects}.
\end{rmk}

\begin{ex}
Let $(\frkg,[\cdot,\cdot])$ be a Lie algebra. We define $\Courant{\cdot,\cdot,\cdot
 }:\wedge^2\g\otimes \g\lon \g$ by  $$\Courant{x,y,z}:=[[x,y],z],\quad \forall x,y, z \in \mathfrak{g}.$$  Then $(\g,\Courant{\cdot,\cdot,\cdot})$ becomes a Lie triple system naturally.
\end{ex}

\begin{defi}\cite{Yamaguti3}
Let $(\g,\Courant{\cdot,\cdot,\cdot})$ be a Lie triple system and $V$ a vector space. A {\bf representation of $\g$ on $V$} consists of a bilinear map $\rho:\otimes^2 \g \to \gl(V)$ such that for all $x,y,z,w \in \g$,
\begin{eqnarray}
~&&\label{RYT4}\rho(z,w)\rho(x,y)-\rho(y,w)\rho(x,z)-\rho(x,\Courant{y,z,w})+D_\rho(y,z)\rho(x,w)=0,\\
~&&\label{RLY5}\rho(\Courant{x,y,z},w)+\rho(z,\Courant{x,y,w})=[D_\rho(x,y),\rho(z,w)],
\end{eqnarray}
where the bilinear map $D_\rho:\otimes^2\g \to \gl(V)$ is given by
\begin{eqnarray}\label{rep}
D_\rho(x,y):=\rho(y,x)-\rho(x,y), \quad \forall x,y \in \g.
\end{eqnarray}
It is obvious that $D_\rho$ is skew-symmetric and we write $D$ instead of $D_\rho$ without ambiguities in the sequel. We denote a representation of $\g$ on $V$ by $(V;\rho)$.
\end{defi}
\emptycomment{
\begin{rmk}
By \eqref{RLY5} and \eqref{rep}, one can by a direct computation deduce that
\begin{eqnarray}
[D(x,y),D(z,w)]=D(\Courant{x,y,z},w)+D(z,\Courant{x,y,w}), \quad \forall x,y,z,w \in \g.\label{rep3}
\end{eqnarray}
\end{rmk}
}

\begin{ex}\label{ad}
Let $(\g,[\cdot,\cdot,\cdot])$ be a Lie triple system. We define $\frkR :\otimes^2\g \to \gl(\g)$ by $(x,y) \mapsto \frkR_{x,y}$ , where $\frkR_{x,y}z=\Courant{z,x,y}$ for all $z \in \g$. Then $(\g;\frkR)$ forms a representation of $\g$ on itself, called the {\bf adjoint representation}. By \eqref{LY1}, $\frkL\triangleq D_\frkR=\frkR_{y,x}-\frkR_{x,y}$ is given by for all $x,y \in \g$,
\begin{eqnarray*}
\frkL_{x,y}z=\Courant{x,y,z}, \quad \forall z \in \g.\label{lef}
\end{eqnarray*}
\end{ex}

By a direct calculation, we have the following
\begin{pro}\label{semidirect}
Let $(\g,\Courant{\cdot,\cdot,\cdot})$ be a Lie triple system and $V$ a vector space. Let $\rho:\otimes^2 \g \to \gl(V)$ be a linear map. Then $(V;\rho)$ is a representation of $(\g,\Courant{\cdot,\cdot,\cdot})$ if and only if there is a Lie triple system structure $\Courant{\cdot,\cdot,\cdot}_{\rho}$ on the direct sum $\g \oplus V$ which is defined by for all $x,y,z \in \g, u,v,w \in V$,
\begin{eqnarray*}
~\Courant{x+u,y+v,z+w}_{\rho}&=&\Courant{x,y,z}+D(x,y)w+\rho(y,z)u-\rho(x,z)v,
\end{eqnarray*}
where $D$ is given by \eqref{rep}.
This Lie triple system $(\g \oplus V,\Courant{\cdot,\cdot,\cdot}_{\mu})$ is called the {\bf semidirect product Lie triple system}, denoted by $\g \ltimes_{\rho} V$, or simply by $\g \ltimes V$.
\end{pro}

\emptycomment{
Let $(V;\rho)$ be a representation of a Lie triple system $(\g,\Courant{\cdot,\cdot,\cdot})$. The $n+1$-cochain consists of the elements $f$ in
\begin{eqnarray*}
C^{n}(\g;V):=\Hom(\otimes^{2n+1} \g,V)=\Hom(\underbrace{(\ot^2\g)\ot\cdots\ot(\ot^2\g))}_n\ot\g,V), \quad n \geqslant 0.
\end{eqnarray*}
If $n\geqslant 1,$ $f \in C^{n}(\g, V)$ satisfies
\begin{eqnarray}
f(\huaX_1,\cdots ,\huaX_{n-1},x_n,y_n,z)+f(\huaX_1,\cdots,\huaX_{n-1} ,y_n,z,x_n)+f(\huaX_1,\cdots,\huaX_{n-1}, z,x_n,y_n)=0,
\end{eqnarray}
for all $\huaX_i=x_i\ot y_i\in\ot^2\g, \quad 1 \leqslant i \leqslant n.$

For any $f \in C^n(\g, V),~ (n\geqslant 0)$, the coboundary map $\de :C^n(\g;V)\to C^{n+1}(\g;V)$ is defined by
\begin{eqnarray*}
~ &&(\de f)(\huaX_1,\cdots,\huaX_n,z)\\
~ &=&(-1)^n\big(\rho(y_{n},z)f(\huaX_1,\cdots,\huaX_{n-1},x_n)-\rho(x_n,z)f(\huaX_1,\cdots,\huaX_{n-1},y_n)\big)\\
~ &&+\sum_{i=1}^n (-1)^{i+1}D(\huaX_i)f(\huaX_1,\cdots,\hat{\huaX_{i}},\cdots,\huaX_n,z)\\
~ &&+\sum_{i=1}^n (-1)^{i}f(\huaX_1,\cdots,\hat{\huaX_{i}},\cdots,\huaX_n,\Courant{x_i,y_i,z})\\
~ &&+\sum_{j<k} (-1)^j f(\huaX_1,\cdots,\hat{\huaX_{j}},\cdots,\huaX_{k-1},[\huaX_j,\huaX_k]_F,\huaX_{k+1},\cdots,\huaX_n,z),
\end{eqnarray*}
for all $\huaX_i=x_i\ot y_i\in\ot^2\g,~1=1,2,\cdots,n$ and $z\in \g$. Yamaguti showed that $\de^2=0$ and thus $(\oplus_{n=1}^{\infty}(\g;V),\de)$ is a cochain complex.}

In the following, we recall the controlling algebras of Lie triple systems. Let $\g$ be a vector space and $C^*(\g,\g)=\oplus_{p\geqslant0}C^p(\g,\g)$, where $C^p(\g,\g)=\Hom(\otimes^{2p+1}\g,\g)$ and degrees of elements in $C^p(\g,\g)$ are assumed to be $p$. For all $P\in C^p(\g,\g)$ and $Q\in C^q(\g,\g)$, define a graded bracket (called the {\bf Nambu bracket}) $[\cdot,\cdot]_{\mathsf{N}}$
to be
$$[P,Q]_{\mathsf{N}}=P\circ Q-(-1)^{pq}Q\circ P,$$
where $P\circ Q\in C^{p+q}(\g,\g)$ is defined by
{\footnotesize
\begin{eqnarray}
~ &&\nonumber P\circ Q(\huaX_1,\cdots,\huaX_{p+q},x)\\
~ \nonumber&=&\sum_{k=1}^p (-1)^{(k-1)q}(-1)^\sigma\sum_{\sigma\in \mathbb{S}(k-1,q)}P\Big(\huaX_{\sigma(1)},\cdots,\huaX_{\sigma(k-1)},Q\big(\huaX_{\sigma(k)},\cdots,\huaX_{\sigma(k+q-1)},x_{k+q}\big)\otimes y_{k+q}, \huaX_{k+q+1},\cdots,\huaX_{p+q},x\Big)\\
~ &&+\sum_{k=1}^p (-1)^{(k-1)q}(-1)^\sigma\sum_{\sigma\in \mathbb{S}(k-1,q)}P\Big(\huaX_{\sigma(1)},\cdots,\huaX_{\sigma(k-1)},x_{k+q}\otimes Q\big(\huaX_{\sigma(k)},\cdots,\huaX_{\sigma(k+q-1)},y_{k+q}\big), \huaX_{k+q+1},\cdots,\huaX_{p+q},x\Big)\label{control}\\
~ &&\nonumber+\sum_{\sigma\in \mathbb{S}(p,q)}(-1)^\sigma P\Big(\huaX_{\sigma(1)},\cdots,\huaX_{\sigma(p)},Q\big(\huaX_{\sigma(p+1)},\cdots,\huaX_{\sigma(p+q)},x\big)\Big),
\end{eqnarray}}
for all $\huaX_i\in \otimes^2\g,~i=1,2,\cdots,p+q$ and $x\in \g$.

Then $(C^*(\g,\g),[\cdot,\cdot]_{\mathsf{N}})$ is a graded Lie algebra and its Maurer-Cartan elements corresponds to Nambu algebra structures \cite{Rot}. Let $C^*_{\mathsf{LTS}}(\g,\g)=\oplus_{p\geqslant0}C_{\mathsf{LTS}}^p(\g,\g)=\oplus_{p\geqslant0}\Hom_{\mathsf{LTS}}^p(\ot^{2p+1}\g,\g)$ be a graded subspace of $C^*(\g,\g)$ such that for any $P\in C^p_{\mathsf{LTS}}(\g,\g)$, $P$ satisfies
\begin{eqnarray*}
P(\huaX_1,\cdots,\huaX_{p-1},x,x,y)&=&0,\\
~P(\huaX_1,\cdots,\huaX_{p-1},x,y,z)+P(\huaX_1,\cdots,\huaX_{p-1},y,z,x)+P(\huaX_1,\cdots,\huaX_{p-1},z,x,y)&=&0.
\end{eqnarray*}
The corresponding Nambu bracket is denoted by $[\cdot,\cdot]_{\mathsf{LTS}}$, when the graded vector space is restricted  to $C^*_{\mathsf{LTS}}(\g,\g)=\oplus_{p\geqslant0}C_{\mathsf{LTS}}^p(\g,\g)$. Then $(C^*_{\mathsf{LTS}}(\g,\g)=\oplus_{p\geqslant0}C_{\mathsf{LTS}}^p(\g,\g),[\cdot,\cdot]_{\mathsf{LTS}})$ is a graded subalgebra and its Maurer-Cartan elements corresponds to Lie triple system structures \cite{XST}. In the present paper, we always consider the graded Lie algebra $(C^*_{\mathsf{LTS}}(\g,\g),[\cdot,\cdot]_{\mathsf{LTS}})$.

\section{(Proto-, quasi-)twilled Lie triple systems}
In this section, we introduce notions of (proto-, quasi-)twilled Lie triple systems, and use the graded Lie algebra bracket in the controlling algebra to give structures of (proto-, quasi-)twilled Lie triple systems. For this purpose, we have to examine the lift of a given linear map and its bidegree first.
\subsection{Lift and bidegree}
Let $\g_1$ and $\g_2$ be vector spaces, and elements in $\g_1$ are denoted by $x,y,z,x_i$ and those in $\g_2$ by $u,v,w,u_i$. We denote by $\g^{l,k}$ the direct sum of all $(l+k)$-tensor power of $\g_1$ and $\g_2$: $\otimes^{n}(\g_1\oplus\g_2)$, where $n=l+k$, and $l$ (resp. $k$) means the quantities of $\g_1$ (resp. $\g_2$). For example, $\g^{2,1}\subset \ot^3(\g_1\oplus\g_2)$ can be written as
$$\g^{2,1}=(\g_1\ot\g_1\ot\g_2)\oplus(\g_1\ot\g_2\ot\g_1)\oplus(\g_2\ot\g_1\ot\g_1).$$
Then $\displaystyle{\otimes^{n}(\g_1\oplus\g_2)=\oplus_{l+k=n}\g^{l,k}}$. For example,
$$\ot^3(\g_1\oplus\g_2)=\g^{3,0}\oplus\g^{2,1}\oplus\g^{1,2}\oplus\g^{0,3}.$$
By $\Hom$-functor, we have
\begin{eqnarray}
C^p_{\mathsf{LTS}}(\g_1\oplus\g_2,\g_1\oplus\g_2)=\bigoplus_{l+k=2p}C^p_{\mathsf{LTS}}(\g^{l,k},\g_1)\oplus\bigoplus_{l+k=2p}C^p_{\mathsf{LTS}}(\g^{l,k},\g_2).\label{decompose}
\end{eqnarray}
For a linear map $f\in C^p_{\mathsf{LTS}}(\g^{l,k},\g_1)$ (resp. $f\in C^p_{\mathsf{LTS}}(\g^{l,k},\g_2)$), $f$ naturally induces a linear map $\hat{f}\in C^p_{\mathsf{LTS}}(\g_1\oplus\g_2,\g_1\oplus\g_2)$ defined to be
\begin{eqnarray*}
\hat f:=
\begin{cases}
f,\quad \text{on }~~ \g^{l,k},\\
0,\quad \text{all other cases.}
\end{cases}
\end{eqnarray*}
The linear map $\hat f$ is called a {\bf lift} of $f$. For example, the lifts of linear maps $\alpha:\ot^3\g_1\longrightarrow\g_1$, $\beta:\ot^2\g_1\ot\g_2\longrightarrow\g_2$ and $\gamma:\g_2\ot(\ot^2\g_1)\longrightarrow\g_2$ are defined to be
\begin{eqnarray*}
\hat\alpha((x,u),(y,v),(z,w))&=&(\alpha(x,y,z),0),\\
\hat\beta((x,u),(y,v),(z,w))&=&(0,\beta(x,y,w)),\\
\hat\gamma((x,u),(y,v),(z,w))&=&(0,\gamma(u,y,z))
\end{eqnarray*}
respectively. A linear map $H:\g_2\longrightarrow\g_1$ induces its lift given by $\hat{H}(x,u)=(H(u),0)$. It is straightforward to see that $\hat H\circ \hat H=0$.

\vspace{3mm}
Now we give the notion of bidegree of linear maps in $C^p_{\mathsf{LTS}}(\g_1\oplus\g_2,\g_1\oplus\g_2)$ in the sequel.
\begin{defi}
Let $f\in C^p_{\mathsf{LTS}}(\g_1\oplus\g_2,\g_1\oplus\g_2)$ be a linear map and $l,k\in \mathbb Z$. Suppose that $l$ and $k$ satisfies the following conditions:
\begin{itemize}
\item[(i)] $l+k=2p$;
\item[(ii)] If $X\in \g^{l+1,k}$, then $f(X)\in \g_1$;
\item[(iii)] If $X\in \g^{l,k+1}$, then $f(X)\in \g_2$;
\item[(iv)] All the other cases, $f(X)=0$,
\end{itemize}
then we say that the {\bf bidegree} of $f$ is $l|k$, denoted by $||f||=l|k$. A linear map $f$ is said to be {\bf homogeneous} if $f$ has a bidegree.
\end{defi}

Notice that $l+k\geqslant 0,~k,l\geqslant -1$ since $p\geqslant 0$ and $l+1,k+1\geqslant0$. For example, $||\hat H||=-1|1$, where $H:\g_2\longrightarrow\g_1$ is a linear map; $||\hat\alpha||=||\hat\beta||=||\hat\gamma||=2|0$. Consequently, we obtain a homogeneous element $\hat\mu:=\hat\alpha+\hat\beta+\hat\gamma$ whose bidegree is $2|0$:
$$\hat\mu((x,u),(y,v),(z,w))=(\alpha(x,y,z),\beta(x,y,w)+\gamma(u,y,z)-\gamma(v,x,z)).$$
However, there does not exist any linear map whose lift is $\hat\mu$, but we focus on $\hat\mu$ since it is a multiplication of the semidirect product type.

We give some lemmas in the following.

\begin{lem}\label{lemma1}
Let $f_1,\cdots,f_n\in C^p_{\mathsf{LTS}}(\g_1\oplus\g_2,\g_1\oplus\g_2)$ be homogeneous linear maps and the bidegrees of $f_i$ are different. Then $\sum_{i=1}^nf_i=0$ if and only if $f_i=0,~i=1,2,\cdots,n$.
\end{lem}

\begin{lem}\label{lem:0}
If $||f||=-1|l$ (resp. $l|-1$) and $||g||=-1|k$ (resp. $||g||=k|-1$), then $[f,g]_{\mathsf{LTS}}=0$.
\end{lem}
\begin{proof}
Suppose that $||f||=-1|l$ and $||g||=-1|k$. Then both $f$ and $g$ are lifts of some linear maps in $C^*_{\mathsf{LTS}}(\g_2,\g_1)$. Hence, we have $f\circ g=g\circ f=0$, which leads to $[f,g]_{\mathsf{LTS}}=0$.
\end{proof}

\begin{lem}\label{lemma}
Let $f\in C^p_{\mathsf{LTS}}(\g_1\oplus\g_2,\g_1\oplus\g_2)$ and $g\in C^q_{\mathsf{LTS}}(\g_1\oplus\g_2,\g_1\oplus\g_2)$ be homogeneous linear maps whose bidegrees are $l_f|l_f$ and $l_g|k_g$ respectively. Then $f\circ g\in C^{p+q}_{\mathsf{LTS}}(\g_1\oplus\g_2,\g_1\oplus\g_2)$ is a homogeneous linear map of bidegree $l_f+l_g|k_f+k_g$.
\end{lem}

\begin{lem}\label{lem:bidegree}
Suppose that $||f||=l_f|k_f$ and $||g||=l_g|k_g$, then $||[f,g]_{\mathsf{LTS}}||=l_f+l_g|k_f+k_g$.
\end{lem}
\begin{proof}
Since $[P,Q]_{\mathsf{LTS}}=P\circ Q-(-1)^{pq}Q\circ P,$ thus by Lemma \ref{lemma}, we obtain that  $||[f,g]_{\mathsf{LTS}}||=l_f+l_g|k_f+k_g$.
\end{proof}

\subsection{Proto-twilled and quasi-twilled Lie triple systems}
\begin{defi}
Let $(\huaG,\Courant{\cdot,\cdot,\cdot}_{\huaG})$ be a Lie triple system with a decomposition of two subspaces: $\huaG=\g_1\oplus\g_2$. Then we call the triple $(\huaG,\g_1,\g_2)$ a {\bf proto-twilled Lie triple system}.
\end{defi}

It is not necessary that $\g_1$ and $\g_2$ be subalgebras of $\huaG$ in a proto-twilled Lie triple system $(\huaG,\g_1,\g_2)$.

\begin{lem}\label{lem:decom}
Any $1$-cochain $\Theta\in C^1_{\mathsf{LTS}}(\huaG,\huaG)$ can be decomposed into five linear maps whose bidegrees are $3|-1,~2|0,~1|1,~0|2$ and $-1|3$:
$$\Theta=\hat\phi_1+\hat\mu_1+\hat\psi+\hat\mu_2+\hat\phi_2.$$
\end{lem}
\begin{proof}
Notice that $C^1_{\mathsf{LTS}}(\huaG,\huaG)=\Hom_{\mathsf{LTS}}(\ot^3\huaG,\huaG)$, and by \eqref{decompose}, we have
$$C^1_{\mathsf{LTS}}(\huaG,\huaG)=(3|-1)\oplus(2|0)\oplus(1|1)\oplus(0|2)\oplus(-1|3),$$
where $(l|k)$ denotes the space of linear maps of bidegree $l|k$ such that $l+k=2$. By Lemma \ref{lemma1}, $\Theta$ can be decomposed into five homogeneous linear maps of bidegrees $3|-1,~2|0,~1|1,~0|2$ and $-1|3$. The proof is finished.
\end{proof}

The operation $\Courant{(x,u),(y,v),(z,w)}_\huaG$ of $\huaG$ is uniquely decomposed into 12 multiplications by the canonical projections $\huaG\longrightarrow\g_1$ and $\huaG\longrightarrow\g_2$:
\begin{eqnarray*}
\Courant{x,y,z}_\huaG&=&(\Courant{x,y,z}_1,\Courant{x,y,z}_2), \quad \Courant{x,y,w}_\huaG=(\Courant{x,y,w}_1,\Courant{x,y,w}_2),\\
~\Courant{x,v,z}_\huaG&=&(\Courant{x,v,z}_1,\Courant{x,v,z}_2), \quad \Courant{x,v,w}_\huaG=(\Courant{x,v,w}_1,\Courant{x,v,w}_2),\\
~\Courant{u,v,z}_\huaG&=&(\Courant{u,v,z}_1,\Courant{u,v,z}_2),  \quad \Courant{u,v,w}_\huaG=(\Courant{u,v,w}_1,\Courant{u,v,z}_2).
\end{eqnarray*}
Here the operation $\Courant{\cdot,\cdot,\cdot}_1$ (resp. $\Courant{\cdot,\cdot,\cdot}_2$) denotes the projection of $\huaG$ onto $\g_1$ (resp. $\g_2$).

In the following, we use the notation $\Theta$ to denote the operation $\Courant{\cdot,\cdot,\cdot}_\huaG$ on $\huaG$, i.e.,
$$\Theta((x,u),(y,v),(z,w))=\Courant{(x,u),(y,v),(z,w)}_\huaG.$$

Write $\Theta=\hat\phi_1+\hat\mu_1+\hat\psi+\hat\mu_2+\hat\phi_2$ as in Lemma \ref{lem:decom}, then we obtain that
{\footnotesize
\begin{eqnarray}\label{equi}
\left\{\begin{array}{rcl}
~~\hat\phi_1((x,u),(y,v),(z,w))&=&(0,\Courant{x,y,z}_2),\\
~~\hat\mu_1((x,u),(y,v),(z,w))&=&(\Courant{x,y,z}_1,\Courant{x,y,w}_2+\Courant{u,y,z}_2-\Courant{v,x,z}_2),\\
~~\hat\psi((x,u),(y,v),(z,w))&=&(\Courant{x,y,w}_1+\Courant{u,y,z}_1-\Courant{v,x,z}_1,\Courant{u,v,z}_2+\Courant{x,v,w}_2-\Courant{y,u,w}_2),\\
~~\hat\mu_2((x,u),(y,v),(z,w))&=&(\Courant{u,v,z}_1+\Courant{x,v,w}_1-\Courant{y,u,w}_1,\Courant{u,v,w}_2),\\
~~\hat\phi_2((x,u),(y,v),(z,w))&=&(\Courant{u,v,w}_1,0).
\end{array}\right.
\end{eqnarray}}

It is easy to see that $\hat\phi_1$ and $\hat\phi_2$ are lifts of linear maps $\phi_1\in C^1_{\mathsf{LTS}}(\g_1,\g_2)$ and $\phi_2\in C^1_{\mathsf{LTS}}(\g_2,\g_1)$ respectively, where $\phi_1(x,y,z):=\Courant{x,y,z}_2$ and $\phi_2(u,v,w):=\Courant{u,v,w}_1$.

\begin{pro}
The Maurer-Cartan equation $[\Theta,\Theta]_{\mathsf{LTS}}=0$ is equivalent to the following conditions:
\begin{eqnarray*}
\left\{\begin{array}{rcl}
[\hat\phi_1,\hat\mu_1]_{\mathsf{LTS}}&=&0,\\
~[\hat\psi,\hat\phi_1]_{\mathsf{LTS}}+\half[\hat\mu_1,\hat\mu_1]_{\mathsf{LTS}}&=&0,\\
~[\hat\phi_1,\hat\mu_2]_{\mathsf{LTS}}+[\hat\psi,\hat\mu_1]_{\mathsf{LTS}}&=&0,\\
~[\hat\psi,\hat\phi_2]_{\mathsf{LTS}}+[\hat\mu_1,\hat\mu_2]_{\mathsf{LTS}}+\half[\hat\psi,\hat\psi]_{\mathsf{LTS}}&=&0,\\
~[\hat\mu_1,\hat\phi_2]_{\mathsf{LTS}}+[\hat\psi,\hat\mu_2]_{\mathsf{LTS}}&=&0,\\
~[\hat\psi,\hat\phi_2]_{\mathsf{LTS}}+\half[\hat\mu_2,\hat\mu_2]_{\mathsf{LTS}}&=&0,\\
~[\hat\mu_2,\hat\phi_2]_{\mathsf{LTS}}&=&0.
\end{array}\right.
\end{eqnarray*}
\end{pro}
\begin{proof}
Since $\Theta=\hat\phi_1+\hat\mu_1+\hat\psi+\hat\mu_2+\hat\phi_2\in C^1_{\mathsf{LTS}}(\huaG,\huaG)$, then expanding the expression $[\Theta,\Theta]_{\mathsf{LTS}}\in C^2_{\mathsf{LTS}}(\huaG,\huaG)$ yields the following nonzero terms by bidegree classification:
\begin{eqnarray*}
\left\{\begin{array}{rcl}
[\hat\phi_1,\hat\mu_1]_{\mathsf{LTS}}&\in& (5|-1),\\
~2[\hat\psi,\hat\phi_1]_{\mathsf{LTS}}+[\hat\mu_1,\hat\mu_1]_{\mathsf{LTS}}&\in& (4|0),\\
~2[\hat\phi_1,\hat\mu_2]_{\mathsf{LTS}}+2[\hat\psi,\hat\mu_1]_{\mathsf{LTS}}&\in& (3|1),\\
~2[\hat\psi,\hat\phi_2]_{\mathsf{LTS}}+2[\hat\mu_1,\hat\mu_2]_{\mathsf{LTS}}+[\hat\psi,\hat\psi]_{\mathsf{LTS}}&\in& (2|2),\\
~2[\hat\mu_1,\hat\phi_2]_{\mathsf{LTS}}+2[\hat\psi,\hat\mu_2]_{\mathsf{LTS}}&\in& (1|3),\\
~2[\hat\psi,\hat\phi_2]_{\mathsf{LTS}}+[\hat\mu_2,\hat\mu_2]_{\mathsf{LTS}}&\in& (0|4)\\
~[\hat\mu_2,\hat\phi_2]_{\mathsf{LTS}}&\in& (-1|5).
\end{array}\right.
\end{eqnarray*}
Thus the Maurer-Cartan equation $[\Theta,\Theta]_{\mathsf{LTS}}=0$ holds if and only if all of its expanding parts vanish.
\end{proof}

\begin{defi}\label{defi:quasi}
Let $(\huaG,\g_1,\g_2)$ be a proto-twilled Lie triple system and $\Theta=\hat\phi_1+\hat\mu_1+\hat\psi+\hat\mu_2+\hat\phi_2\in C^1_{\mathsf{LTS}}(\huaG,\huaG)$ its structure. If $\phi_2=0$, or equivalently, $\g_2$ is a subalgebra, then the proto-twilled Lie triple system $(\huaG,\g_1,\g_2)$ is called a {\bf quasi-twilled Lie triple system}.
\end{defi}

\begin{rmk}
Since $\huaG=\g_1\oplus\g_2=\g_2\oplus\g_1$, thus the condition in Definition \ref{defi:quasi} is adapted in the case that $\phi_1=0$. Hence the condition making the proto-twilled Lie triple system $(\huaG,\g_1,\g_2)$ into a quasi-twilled Lie triple system is that either $\g_1$ or $\g_2$ is a sublagebra. In the following, we always assume that $\g_2$ is a subalgebra when quasi-twilled Lie triple system is referred.
\end{rmk}

It is direct to obtain the following proposition.
\begin{pro}
The triple $(\huaG,\g_1,\g_2)$ is a quasi-twilled Lie triple system if and only if the following conditions are satisfied:
\begin{eqnarray*}
\left\{\begin{array}{rcl}
[\hat\phi_1,\hat\mu_1]_{\mathsf{LTS}}&=&0,\\
~[\hat\psi,\hat\phi_1]_{\mathsf{LTS}}+\half[\hat\mu_1,\hat\mu_1]_{\mathsf{LTS}}&=&0,\\
~[\hat\phi_1,\hat\mu_2]_{\mathsf{LTS}}+[\hat\psi,\hat\mu_1]_{\mathsf{LTS}}&=&0,\\
~[\hat\mu_1,\hat\mu_2]_{\mathsf{LTS}}+\half[\hat\psi,\hat\psi]_{\mathsf{LTS}}&=&0,\\
~[\hat\psi,\hat\mu_2]_{\mathsf{LTS}}&=&0,\\
~\half[\hat\mu_2,\hat\mu_2]_{\mathsf{LTS}}&=&0.
\end{array}\right.
\end{eqnarray*}
\end{pro}

\subsection{Twilled Lie triple systems and $L_\infty$-algebras}
In this subsection, we introduce the notion of twilled Lie triple systems, which is our main object in the present paper, and then we construct an $L_\infty$-algebra from a twilled Lie triple system via the higher derived brackets.
\begin{defi}\label{defi:twilled}
Let $(\huaG,\g_1,\g_2)$ be a proto-twilled Lie triple system and $\Theta=\hat\phi_1+\hat\mu_1+\hat\psi+\hat\mu_2+\hat\phi_2\in C^1_{\mathsf{LTS}}(\huaG,\huaG)$ its structure. If both $\phi_1$ and $\phi_2$ all vanish, or equivalently, both $\g_1$ and $\g_2$ are subalgebras of $\huaG$, then the proto-twilled Lie triple system $(\huaG,\g_1,\g_2)$ is called a {\bf twilled Lie triple system}.
\end{defi}

\begin{rmk}
In \cite{Sheng Zhao}, the first author et al. explored product structures on Lie-Yamaguti algebras, and we obtain the corresponding product structure on Lie triple system if we restrict the binary operation on Lie-Yamaguti algebra to zero. By Definition \ref{defi:twilled}, we see that $\huaG$ is a twilled Lie triple system if and only if there exists a product structure on $\huaG$.
\end{rmk}

It is easy to deduce the following proposition.
\begin{pro}\label{pro:twilled}
The triple $(\huaG,\g_1,\g_2)$ is a twilled Lie triple system if and only if the following conditions are satisfied:
\begin{eqnarray*}
\left\{\begin{array}{rcl}
~\half[\hat\mu_1,\hat\mu_1]_{\mathsf{LTS}}&=&0,\\
~[\hat\psi,\hat\mu_1]_{\mathsf{LTS}}&=&0,\\
~[\hat\mu_1,\hat\mu_2]_{\mathsf{LTS}}+\half[\hat\psi,\hat\psi]_{\mathsf{LTS}}&=&0,\\
~[\hat\psi,\hat\mu_2]_{\mathsf{LTS}}&=&0,\\
~\half[\hat\mu_2,\hat\mu_2]_{\mathsf{LTS}}&=&0.
\end{array}\right.
\end{eqnarray*}
\end{pro}

\begin{rmk}
From Proposition \ref{pro:twilled}, we obtain that $\hat\mu_1\in C^1_{\mathsf{LTS}}(\huaG,\huaG)$ is a Lie triple system structure on $\huaG=\g_1\oplus\g_2$. Moreover, by \eqref{equi}, we see that $\hat\mu_1|_{\g_1\ot\g_1}$ is a Lie triple system structure on $\g_1$, and we denote this Lie triple system by $(\g_1,\Courant{\cdot,\cdot,\cdot}_{\g_1})$. Set $\rho_1(x,y)u:=\hat\mu_1(u,x,y)$, then $D_1(x,y)u=D_{\rho_1}(x,y)u=\hat\mu_1(x,y,u)$ since $\g_2$ is a Lie triple system. Consequently, $(\g_2;\rho_1)$ is a representation of $\g_1$. Similarly, we have $\hat\mu_2|_{\g_1\ot\g_2}$ is a Lie triple system structure on $\g_2$, and $(\g_1;\rho_2)$ is a representation of $\g_2$, where $\rho_2(u,v)x:=\hat\mu_2(x,u,v)$ and $D_2(u,v)x=D_{\rho_2}(u,v)x=\hat\mu_2(u,v,x)$.
\end{rmk}

\emptycomment{
\begin{defi}
An $L_\infty$-algebra is a pair $(\g,\{l_k\}_{k=1}^\infty)$, where $\g=\oplus_{k\in \mathbb Z}\g_k$ is a $\mathbb Z$-graded vector space, the degrees of whose elements in $\g_k$ are assumed to be $k$, and $\{l_k\}_{k=1}^\infty$ is a collection of graded linear maps $l_k:\ot^k\g\longrightarrow\g,~(k\geqslant 1)$ of degree $1$ such that the following two conditions are satisfied for any homogeneous elements $x_1,\cdots,x_n\in \g$,
\begin{itemize}
\item[(i)] (graded symmetry) for all $\sigma\in S_n$, we have
$$l_n(x_{\sigma(1)},\cdots,x_{\sigma(n)})=\varepsilon(\sigma)l_n(x_1,\cdots,x_n),$$
\item[(ii)] (graded Jacobi identity) for all $n\geqslant 1$,
$$\sum_{i=1}^n\sum_{\sigma\in \mathbb S_{(i,n-i)}}\varepsilon(\sigma)l_{n-i+1}\Big(l_i(x_{\sigma(1)},\cdots,x_{\sigma(i)}),x_{\sigma(i+1)},\cdots,x_{\sigma(n)}\Big)=0.$$
\end{itemize}
Here $\mathbb S_{(i,n-i)}$ stands for the collection of all $(i,n-i)$-shuffles and $\varepsilon(\mathfrak{\sigma})$ stands for the Koszul sign: switching any two successive elements $x_i$ and $x_{i+1}$ leads to a sign change $(-1)^{|x_i||x_{i+1}|}$.
\end{defi}

\begin{defi}
Let $(\g,\{l_k\}_{k=1}^\infty)$ be an $L_\infty$-algebra. An element $\alpha\in \g_0$ of degree $0$ is called a {\bf Maurer-Cartan element } if $\alpha$ satisfies the following Maurer-Cartan equation
$$\sum_{k=1}^\infty\frac{1}{k!}l_k(\alpha,\cdots,\alpha)=0.$$
\end{defi}
}

In the sequel, we construct an $L_\infty$-algebra via a given twilled Lie triple system. Before this, let us recall the method of higher derived brackets of a differential graded Lie algebra. Let $(\g,[\cdot,\cdot],d)$ be a differential graded Lie algebra and $d_t:=\sum_{i=0}^\infty d_it^i$ a formal deformation of $d$ with $d_0=d$. Then $d_t$ is a differential on $\g[[t]]$, which is a Lie triple system of formal series with coefficients in $\g$. The square zero condition $d_t\circ d_t=0$ is equivalent to
$$\sum_{i+j=n,\atop i,j\geqslant 0}d_i\circ d_j=0,\quad n\in \mathbb{Z}_{\geqslant 0}.$$

Then set $l_k:\ot^k\g\longrightarrow\g$ to be
\begin{eqnarray}
l_k(x_1,\cdots,x_k)=[\cdots[[d_{k-1}(x_1),x_2],x_3],\cdots,x_k]\label{multi}
\end{eqnarray}
for all homogeneous elements $x_1,\cdots,x_k\in \g$. The collection of graded linear maps $\{l_k\}_{k=1}^\infty$ is called the {\bf higher derived brackets}.

A suitable $L_\infty$-algebra can be obtained from a differential graded Lie algebra via the method of higher derived brackets. The notion of $L_\infty$-algebras was introduced in \cite{Stasheff}. See \cite{LM,LS} for more details.
\begin{pro}(\cite{Uchino})\label{derived}
Let $(\g,[\cdot,\cdot],d)$ be a differential graded Lie algebra, and $\h\subset \g$ an abelian subalgebra, i.e., $[\h,\h]=0$. If $l_k$ is closed on $\h$, where $l_k$ is given by \eqref{multi}, then $(\h,\{l_k\}_{i=1}^\infty)$ is an $L_\infty$-algebra.
\end{pro}

Let $(\huaG,\g_1,\g_2)$ be a twilled Lie triple system, and $C^p_{\mathsf{LTS}}(\g_2,\g_1)=\Hom_{\mathsf{LTS}}(\ot^{2p+1}\g_2,\g_1)$ a subspace of $(C^*_{\mathsf{LTS}}(\huaG,\huaG),[\cdot,\cdot]_{\mathsf{LTS}})$.  For any $f\in C^p_{\mathsf{LTS}}(\g_2,\g_1),~g\in C^q_{\mathsf{LTS}}(\g_2,\g_1),~h\in C^r_{\mathsf{LTS}}(\g_2,\g_1)$, define
\begin{eqnarray*}
&&l_1:C^p_{\mathsf{LTS}}(\g_2,\g_1)\longrightarrow C^{p+1}_{\mathsf{LTS}}(\g_2,\g_1),\\
~&&l_2:C^p_{\mathsf{LTS}}(\g_2,\g_1)\times C^q_{\mathsf{LTS}}(\g_2,\g_1)\longrightarrow C^{p+q+1}_{\mathsf{LTS}}(\g_2,\g_1),\\
~&&l_3:C^p_{\mathsf{LTS}}(\g_2,\g_1)\times C^q_{\mathsf{LTS}}(\g_2,\g_1)\times C^r_{\mathsf{LTS}}(\g_2,\g_1)\longrightarrow C^{p+q+r+1}_{\mathsf{LTS}}(\g_2,\g_1)
\end{eqnarray*}
to be
\begin{eqnarray*}
l_1(f)&=&[\hat\mu_2,\hat f]_{\mathsf{LTS}},\\
~l_2(f,g)&=&[[\hat\psi,\hat f]_{\mathsf{LTS}},\hat g]_{\mathsf{LTS}},\\
~l_3(f,g,h)&=&[[[\hat\mu_1,\hat f]_{\mathsf{LTS}},\hat g]_{\mathsf{LTS}},\hat h]_{\mathsf{LTS}}
\end{eqnarray*}
respectively. Here $\hat f,~\hat g$ and $\hat h$ are lifts of linear maps $f,~g$ and $h$ respectively.

\begin{thm}\label{Linfty}
Let $(\huaG,\g_1,\g_2)$ be a twilled Lie triple system. Then with the above notations, $(C^*_{\mathsf{LTS}}(\g_2,\g_1),l_1,l_2,l_3)$ is an $L_\infty$-algebra.
\end{thm}
\begin{proof}
We define derivations $d_i$ to be
$$d_0:=[\hat\mu_2,\cdot]_{\mathsf{LTS}},\quad d_1:=[\hat\psi,\cdot]_{\mathsf{LTS}},\quad d_2:=[\hat\mu_1,\cdot]_{\mathsf{LTS}},\quad d_k=0, \quad \forall k\geqslant3.$$
For all $f\in C^*_{\mathsf{LTS}}(\huaG,\huaG)$, since $2[[\hat\mu_2,\hat\mu_2]_{\mathsf{LTS}},f]_{\mathsf{LTS}}=[[\hat\mu_2,\hat\mu_2]_{\mathsf{LTS}},f]_{\mathsf{LTS}}$, we deduce that
$$d_0\circ d_0=0,$$
which implies that $(C^*_{\mathsf{LTS}}(\huaG,\huaG),[\cdot,\cdot]_{\mathsf{LTS}},d_0)$ is a differential graded Lie algebra.
Moreover, for all $f\in C^*_{\mathsf{LTS}}(\huaG,\huaG)$, by Proposition \ref{pro:twilled}, we have
\begin{eqnarray*}
(d_0\circ d_1+d_1\circ d_0)(f)&=&[\hat\mu_2,[\hat\psi,f]_{\mathsf{LTS}}]_{\mathsf{LTS}}+[\hat\psi,[\hat\mu_2,f]_{\mathsf{LTS}}]_{\mathsf{LTS}}\\
~&=&[[\hat\mu_2,\hat\psi]_{\mathsf{LTS}},f]_{\mathsf{LTS}}=0,\\
~(d_0\circ d_2+d_1\circ d_1+d_2\circ d_0)(f)&=&[\hat\mu_2,[\hat\mu_1,f]_{\mathsf{LTS}}]_{\mathsf{LTS}}+[\hat\psi,[\hat\psi,f]_{\mathsf{LTS}}]_{\mathsf{LTS}}+[\hat\mu_1,[\hat\mu_2,f]_{\mathsf{LTS}}]_{\mathsf{LTS}}\\
~&=&[[\hat\mu_1,\hat\mu_2]_{\mathsf{LTS}},f]_{\mathsf{LTS}}+\half[[\hat\psi,\hat\psi]_{\mathsf{LTS}},f]_{\mathsf{LTS}}=0,\\
~(d_1\circ d_2+d_2\circ d_1)(f)&=&[\hat\psi,[\hat\mu_1,f]_{\mathsf{LTS}}]_{\mathsf{LTS}}+[\hat\mu_1,[\hat\psi,f]_{\mathsf{LTS}}]_{\mathsf{LTS}}\\
~ &=&[[\hat\psi,\hat\mu_1]_{\mathsf{LTS}},f]_{\mathsf{LTS}}=0,\\
~ (d_2\circ d_2)(f)&=&[\hat\mu_1,[\hat\mu_1,f]_{\mathsf{LTS}}]_{\mathsf{LTS}}=\half[[\hat\mu_1,\hat\mu_1]_{\mathsf{LTS}},f]_{\mathsf{LTS}}=0.
\end{eqnarray*}
Thus we obtain that $\displaystyle{\sum_{i+j=n,\atop n\geqslant 0}d_i\circ d_j=0}$ and we can construct the higher derived brackets on $C^*_{\mathsf{LTS}}(\huaG,\huaG)$ as follows:
\begin{eqnarray*}
l_1(f)&=&[\hat\mu_2,f]_{\mathsf{LTS}},\\
~l_2(f,g)&=&[[\hat\psi,f]_{\mathsf{LTS}},g]_{\mathsf{LTS}},\\
~l_3(f,g,h)&=&[[[\hat\mu_1,f]_{\mathsf{LTS}},g]_{\mathsf{LTS}},h]_{\mathsf{LTS}},\\
~ l_k&=&0,\quad k\geqslant 4,
\end{eqnarray*}
for all $f\in C^p_{\mathsf{LTS}}(\huaG,\huaG),~g\in C^q_{\mathsf{LTS}}(\huaG,\huaG),~h\in C^r_{\mathsf{LTS}}(\huaG,\huaG)$. It is direct to see that $C^*_{\mathsf{LTS}}(\g_2,\g_1)$ is an abelian subalgebra of $(C^*_{\mathsf{LTS}}(\huaG,\huaG),[\cdot,\cdot]_{\mathsf{LTS}})$, and we show that $l_1,l_2,l_3$ are closed on $C^*_{\mathsf{LTS}}(\g_2,\g_1)$ in the sequel.  For all $f\in C^p_{\mathsf{LTS}}(\g_2,\g_1)$, we have $||\hat f||=-1|2p+1$. Then by Lemma \ref{lem:bidegree}, we deduce that $||l_1(f)||=||[\hat\mu_2,\hat f]_{\mathsf{LTS}}||=-1|2p+3$, which implies that $l_1(f)\in C^*_{\mathsf{LTS}}(\g_2,\g_1)$. Similarly, for all $g\in C^q_{\mathsf{LTS}}(\g_2,\g_1)$ and $h\in C^r_{\mathsf{LTS}}(\g_2,\g_1)$, we have
$$||l_2(f,g)||=-1|2p+2q+3,\quad ||l_3(f,g,h)||=-1|2p+2q+2r+3,$$
and thus $l_2(f,g),l_3(f,g,h)\in C^*_{\mathsf{LTS}}(\g_2,\g_1)$, which implies that $l_1,l_2$ and $l_3$ are closed on $C^*_{\mathsf{LTS}}(\g_2,\g_1)$. Consequently, by Proposition \ref{derived}, we obtain that $(C^*_{\mathsf{LTS}}(\g_2,\g_1),l_1,l_2,l_3)$ is an $L_\infty$-algebra.
\end{proof}

\section{Twisting of twilled Lie triple systems}
Let $f\in C^0_{\mathsf{LTS}}(\huaG,\huaG)$ be a $0$-cochain. Set
$$\exp(X_f)(\cdot):=\sum_{k=0}^\infty \frac{1}{k!}X_f^k,$$
where $X_f^k:=\underbrace{[\cdots[[}_k\cdot,f]_{\mathsf{LTS}},f]_{\mathsf{LTS}},\cdots,f]_{\mathsf{LTS}}$. Note that $\exp(X_f)(\cdot)$ is not well defined in general.

Let $(\huaG,\g_1,\g_2)$ be a proto-twilled Lie triple system with the structure $\Theta:=\hat\phi_1+\hat\mu_1+\hat\psi+\hat\mu_2+\hat\phi_2$ and $\hat H\in C^0_{\mathsf{LTS}}(\huaG,\huaG)$ be the lift of a linear map $H:\g_2\longrightarrow\g_1$. Then $\exp(X_{\hat H})(\cdot)$ is a well-defined operator since $\hat H\circ \hat H=0$.

\begin{defi}
With the above notations, the transformation $\exp(X_{\hat H})(\Theta)$ is called a {\bf twisting} of $\Theta$ by $H$, which is denoted by $\Theta^H$.
\end{defi}

We need to show that the twisting of the structure is again a Lie triple system structure. Before this, we show a property of the twisting.
\begin{lem}\label{twit}
We have the following equation:
$$\Theta^H=\exp(-\hat H)\circ \Theta\circ \Big(\exp(\hat H)\ot \exp(\hat H)\ot \exp(\hat H)\Big).$$
\end{lem}
\begin{proof}
For all $(x_1,u_1),(x_2,u_2),(x_3,u_3)\in \huaG$, we have
\begin{eqnarray*}
~ &&[\Theta,\hat H]_{\mathsf{LTS}}\Big((x_1,u_1),(x_2,u_2),(x_3,u_3)\Big)\\
~ &=&(\Theta\circ\hat H-\hat H\circ\Theta)\Big((x_1,u_1),(x_2,u_2),(x_3,u_3)\Big)\\
~ &=&\Theta\bigg(H(x_1,u_1),(x_2,u_2),(x_3,u_3)\bigg)+\Theta\bigg((x_1,u_1),H(x_2,u_2),(x_3,u_3)\bigg)+\Theta\bigg((x_1,u_1),(x_2,u_2),H(x_3,u_3)\bigg)\\
~ &&-\hat H\bigg(\Theta\Big((x_1,u_1),(x_2,u_2),(x_3,u_3)\Big)\bigg)\\
~ &=&\Theta\bigg((H(u_1),0),(x_2,u_2),(x_3,u_3)\bigg)+\Theta\bigg((x_1,u_1),(H(u_2),0),(x_3,u_3)\bigg)+\Theta\bigg((x_1,u_1),(x_2,u_2),(H(u_3),0)\bigg)\\
~ &&-\hat H\bigg(\Theta\Big((x_1,u_1),(x_2,u_2),(x_3,u_3)\Big)\bigg).
\end{eqnarray*}
Thus we obtain that
\begin{eqnarray}\label{1order}
X_{\hat H}(\Theta)=\Theta\circ(\hat H\ot\Id\ot\Id)+\Theta\circ (\Id\ot \hat H\ot\Id)+\Theta\circ(\Id\ot\Id\ot\hat H)-\hat H\circ \Theta.
\end{eqnarray}
By Eq. \eqref{1order} and the fact that $\hat H\circ \hat H=0$, we compute that
\begin{eqnarray*}
~&& [[\Theta,\hat H]_{\mathsf{LTS}},\hat H]_{\mathsf{LTS}}\Big((x_1,u_1),(x_2,u_2),(x_3,u_3)\Big)\\
~ &=&[\Theta,\hat H]_{\mathsf{LTS}}\Big((H(u_1),0),(x_2,u_2),(x_3,u_3)\Big)+[\Theta,\hat H]_{\mathsf{LTS}}\Big((x_1,u_1),(H(u_2),0),(x_3,u_3)\Big)\\
~ &&+[\Theta,\hat H]_{\mathsf{LTS}}\Big((x_1,u_1),(x_2,u_2),(H(u_3),0)\Big)-\hat H\circ [\Theta,\hat H]_{\mathsf{LTS}}\Big((x_1,u_1),(x_2,u_2),(x_3,u_3)\Big)\\
~ &=&2\Theta\Big((H(u_1),0),(H(u_2),0),(x_3,u_3)\Big)+2\Theta\Big((H(u_1),0),(x_2,u_2),(H(u_3),0)\Big)\\
~ &&+2\Theta\Big((x_1,u_1),(H(u_2),0),(H(u_3),0)\Big)-2\hat H\bigg(\Theta\Big((H(u_1),0),(x_2,u_2),(x_3,u_3)\Big)\bigg)\\
~ &&-2\hat H\bigg(\Theta\Big((x_1,u_1),(H(u_2),0),(x_3,u_3)\Big)\bigg)-2\hat H\bigg(\Theta\Big((x_1,u_1),(x_2,u_2),(H(u_3),0)\Big)\bigg).
\end{eqnarray*}
Hence we obtain that
\begin{eqnarray}
\nonumber X_{\hat H}^2(\Theta)&=&2\Theta\circ (\hat H\ot\hat H\ot\Id)+2\Theta\circ (\hat H\ot\Id\ot\hat H)+2\Theta\circ (\Id\ot\hat H\ot\hat H)\\
~ &&-2\hat H\circ\Theta\circ(\hat H\ot\Id\ot\Id)-2\hat H\circ\Theta\circ(\Id \ot\hat H\ot\Id)-2\hat H\circ\Theta\circ(\Id\ot\Id\ot\hat H).
\end{eqnarray}
Similarly, we obtain that
\begin{eqnarray}
X_{\hat H}^3(\Theta)&=&6\Theta\circ(\hat H\ot\hat H\ot\hat H)-6\hat H\circ\Theta\circ(\hat H\ot\hat H\ot\Id)\\
~ &&\nonumber-6\hat H\circ\Theta\circ(\hat H\ot\Id\ot\hat H)-6\hat H\circ\Theta\circ(\Id\ot\Id\ot\hat H),\\
~X_{\hat H}^4(\Theta)&=&-24\hat H\circ\Theta\circ(\hat H\ot\hat H\ot\hat H),\\
~X_{\hat H}^k(\Theta)&=&0,\quad \forall k\geqslant 5.\label{5order}
\end{eqnarray}
By Eqs. \eqref{1order}-\eqref{5order} and by using the fact that $\hat H\circ \hat H=0$ again, we have
\begin{eqnarray*}
~&&\exp(-\hat H)\circ \Theta\circ\Big(\exp(\hat H)\ot \exp(\hat H)\ot \exp(\hat H)\Big)\\
~ &=&\Big(\Id-\hat H\Big)\circ\Theta\circ \Big((\Id+\hat H)\ot(\Id+\hat H)\ot(\Id+\hat H)\Big)\\
~ &=&\Theta\circ(\Id\ot\Id\ot\Id)+\Theta\circ(\Id\ot\hat H\ot\Id)+\Theta\circ(\hat H\ot\Id\ot\Id)+\Theta\circ(\hat H\ot\hat H\ot\Id)\\
~ &&+\Theta\circ(\Id\ot\Id\ot\hat H)+\Theta\circ(\Id\ot\hat H\ot\hat H)+\Theta\circ(\hat H\ot\Id\ot\hat H)+\Theta\circ(\hat H\ot\hat H\ot\hat H)\\
~ &&-\hat H\circ\Theta\circ(\Id\ot\Id\ot\Id)-\hat H\circ\Theta\circ(\Id\ot\hat H\ot\Id)-\hat H\circ\Theta\circ(\hat H\ot\Id\ot\Id)-\hat H\circ\Theta\circ(\hat H\ot\hat H\ot\Id)\\
~ &&-\hat H\circ\Theta\circ(\Id\ot\Id\ot\hat H)-\hat H\circ\Theta\circ(\Id\ot\hat H\ot\hat H)-\hat H\circ\Theta\circ(\hat H\ot\Id\ot\hat H)-\hat H\circ\Theta\circ(\hat H\ot\hat H\ot\hat H)\\
~ &=&\Theta+X_{\hat H}(\Theta)+\half X_{\hat H}^2(\Theta)+\frac{1}{3!}X_{\hat H}^3(\Theta)+\frac{1}{4!}X_{\hat H}^4(\Theta)\\
~ &=&\Theta^H.
\end{eqnarray*}
This completes the proof.
\end{proof}

\begin{pro}
The twisting $\Theta^H$ is a Lie triple system structure on $\huaG$.
\end{pro}
\begin{proof}
For all $(x,u),(y,v),(z,w)\in \huaG$, one has
\begin{eqnarray*}
~ &&\Theta^H((x,u),(x,u),(y,v))\\
~ &=&\exp(-\hat H)\circ \Theta\circ\Big(\exp(\hat H)\ot\exp(\hat H)\ot\exp(\hat H)\Big)((x,u),(x,u),(y,v))\\
~ &=&\Big(\Id-\hat H\Big)\bigg(\Theta\Big((x+H(u),u),(x+H(u),u),(y+H(v),v)\Big)\bigg)=0,
\end{eqnarray*}
and
\begin{eqnarray*}
~ &&\Theta^H((x,u),(y,v),(z,w))+c.p.\\
~ &=&\exp(-\hat H)\bigg(\Theta\Big((x+H(u),u),(y+H(v),v),(z+H(w),w)\Big)+c.p.\bigg)\\
~ &=&0.
\end{eqnarray*}
Here, the last equalities hold since $\Theta$ is a Lie triple system structure. Moreover, one needs to show that $\Theta^{\hat H}$ satisfies the fundamental identity, or equivalently, $\Theta^{\hat H}$ satisfies the following Maurer-Cartan equation:
 $$[\Theta^H,\Theta^H]_{\mathsf{LTS}}=0.$$
 In fact, by Lemma \ref{twit}, we have
\begin{eqnarray*}
~[\Theta^H,\Theta^H]_{\mathsf{LTS}}&=&2\Theta^H\circ\Theta^H\\
~ &=&2\exp(-\hat H)\circ (\Theta\circ\Theta)\circ \Big(\exp(\hat H)\ot \exp(\hat H)\ot \exp(\hat H)\Big)\\
~ &=&\exp(-\hat H)\circ [\Theta,\Theta]_{\mathsf{LTS}}\circ \Big(\exp(\hat H)\ot \exp(\hat H)\ot \exp(\hat H)\Big)=0
\end{eqnarray*}
This completes the proof.
\end{proof}

\begin{cor}
The linear map
$$\exp(\hat H):(\huaG,\Theta^H)\longrightarrow(\huaG,\Theta)$$
is an isomorphism between Lie triple systems.
\end{cor}

Let $(\huaG,\g_1,\g_2)$ be a proto-twilled Lie triple system with its structure $\Theta$. It is easy to see that $(\huaG,\Theta^H)$ is also a proto-twilled Lie triple system, and thus $\Theta^H$ can also be decomposed into 5 terms: $\Theta^H=\hat\phi_1^H+\hat\mu_1^H+\hat\psi^H+\hat\mu_2^H+\hat\phi_2^H$, where the bidegrees of $\hat\phi_1^H,\hat\mu_1^H,\hat\psi^H,\hat\mu_2^H$ and $\hat\phi_2^H$ are $3|-1,2|0,1|1,0|2$ and $-1|3$ respectively. Then the decomposed terms of twisting operations are determined by the following result.
\begin{thm}\label{twilledtwisting}
With the above notations, we have
\begin{eqnarray*}
\left\{\begin{array}{rcl}
\hat\phi_1^H&=&\hat\phi_1,\\
\hat\mu_1^H&=&\hat\mu_1+X_{\hat H}(\hat \phi_1),\\
\hat\psi^H&=&\hat\psi+X_{\hat H}(\hat\mu_1)+\half X_{\hat H}^2(\hat\phi_1),\\
\hat\mu_2^H&=&\hat\mu_2+X_{\hat H}(\hat\psi)+\half X_{\hat H}^2(\hat\mu_1)+\frac{1}{6}X_{\hat H}^3(\hat\phi_1),\\
\hat\phi_2^H&=&\hat\phi_2+X_{\hat H}(\hat\mu_2)+\half X_{\hat H}^2(\hat\psi)+\frac{1}{6}X_{\hat H}^3(\hat\mu_1)+\frac{1}{24}X_{\hat H}^4(\hat\phi_1).
\end{array}\right.
\end{eqnarray*}
\end{thm}
\begin{proof}
By the proof of Lemma \ref{twit}, we have
\begin{eqnarray}
\Theta^H=\Theta+X_{\hat H}(\Theta)+\half X_{\hat H}^2(\Theta)+\frac{1}{3!}X_{\hat H}^3(\Theta)+\frac{1}{4!}X_{\hat H}^4(\Theta).\label{twisting}
\end{eqnarray}
Thus the first term is $\Theta=\hat\phi_1+\hat\mu_1+\hat\psi+\hat\mu_2+\hat\phi_2$. By linearity of $X_{\hat H}(\cdot)$ and by Lemma \ref{lem:0} which leads to $X_{\hat H}(\phi_2)=0$, we have
$$X_{\hat H}(\Theta)=X_{\hat H}(\phi_1)+X_{\hat H}(\mu_1)+X_{\hat H}(\psi)+X_{\hat H}(\mu_2).$$
Here the bidegrees of each terms is $||X_{\hat H}(\phi_1)||=2|0,||X_{\hat H}(\mu_1)||=1|1,||X_{\hat H}(\psi)||=0|2$ and $||X_{\hat H}(\mu_2)||=-1|3$ respectively.
Similarly, the remaining terms are
$$X_{\hat H}^2(\Theta)=X_{\hat H}^2(\phi_1)+X_{\hat H}^2(\mu_1)+X_{\hat H}^2(\psi),$$
where $||X_{\hat H}^2(\phi_1)||=1|1,||X_{\hat H}^2(\mu_1)||=0|2,||X_{\hat H}^2(\psi)||=-1|3$;
$$X_{\hat H}^3(\Theta)=X_{\hat H}^3(\phi_1)+X_{\hat H}^2(\mu_1),$$
where $||X_{\hat H}^3(\phi_1)||=0|2$ and $||X_{\hat H}^2(\mu_1)||=-1|3$;
and
$$X_{\hat H}^4(\Theta)=X_{\hat H}^4(\hat \phi_1)$$ whose bidegree is $||X_{\hat H}^4(\hat \phi_1)||=-1|3$.
Thus by Eq. \eqref{twisting} and the bidegree convention, we obtain that the $3|-1$-component is $\hat\phi_1$, which gives $\hat\phi_1^H$; the sum of all $2|0$-components is $\hat\mu_1+X_{\hat H}(\hat\phi_1)$, which gives $\hat\mu_1^H$; the sum of all $1|1$-components is $\hat\psi+X_{\hat H}(\hat\mu_1)+\half X_{\hat H}^2(\hat\phi_1)$, which gives $\hat\psi^H$; the sum of all $0|2$-components is $\hat\mu_2+X_{\hat H}(\hat\psi)+\half X_{\hat H}(\hat\mu_1)+\frac{1}{6}X_{\hat H}^3(\hat\phi_1)$, which gives $\hat\mu_2^H$; and the sum of all $-1|3$-components is $\hat\phi_2+X_{\hat H}(\hat\mu_2)+\half X_{\hat H}(\hat\psi)+\frac{1}{6}X_{\hat H}^3(\hat\mu_1)+\frac{1}{24}X_{\hat H}^4(\hat\phi_1)$, which gives $\hat\phi_2^H$. This finishes the proof.
\end{proof}

Now we are ready to give the Maurer-Cartan characterizations of twisting of twilled Lie triple systems.
\begin{thm}\label{MC}
Let $((\huaG,\Theta),\g_1,\g_2)$ be a twilled Lie triple system, and $H\in C^0(\g_2,\g_1)$ a linear map. Then the twisting $((\huaG,\Theta^H),\g_1,\g_2)$ is also a twilled Lie triple system if and only if $H$ is a Maurer-Cartan element of the $L_\infty$-algebra $(C^*_{\mathsf{LTS}}(\g_2,\g_1),l_1,l_2,l_3)$ constructed in Theorem \ref{Linfty}, i.e., $H$ satisfies the following Maurer-Cartan equation
$$l_1(H)+\half l_2(H,H)+\frac{1}{3!}l_3(H,H,H)=0.$$
\end{thm}
\begin{proof}
Write $\Theta=\hat\phi_1+\hat\mu_1+\hat\psi+\hat\mu_2+\hat\phi_2$, since $(\huaG,\Theta)$ is a twilled Lie triple system, then we have $\phi_1=\phi_2=0$. By Theorem \ref{twilledtwisting}, we have
\begin{eqnarray*}
\left\{\begin{array}{rcl}
\hat\mu_1^H&=&\hat\mu_1,\\
\hat\psi^H&=&\hat\psi+X_{\hat H}(\hat\mu_1),\\
\hat\mu_2^H&=&\hat\mu_2+X_{\hat H}(\hat\psi)+\half X_{\hat H}^2(\hat\mu_1),\\
\hat\phi_2^H&=&X_{\hat H}(\hat\mu_2)+\half X_{\hat H}^2(\hat\psi)+\frac{1}{6}X_{\hat H}^3(\hat\mu_1).
\end{array}\right.
\end{eqnarray*}
Thus the twisting $((\huaG,\Theta^H),\g_1,\g_2)$ is also a twilled Lie triple system if and only if $\hat\phi_2^H=0$, which implies that $H$ is a Maurer-Cartan element of the $L_\infty$-algebra $(C^*_{\mathsf{LTS}}(\g_2,\g_1),l_1,l_2,l_3)$.
\end{proof}

The twisting of twilled Lie triple systems gives rise to a Lie triple system structure on its decomposed subspace.
\begin{pro}
Let $((\huaG,\Theta),\g_1,\g_2)$ and its twisting $((\huaG,\Theta^H),\g_1,\g_2)$ both be twilled Lie triple systems. Then for all $u,v,w\in \g_2,$
\begin{eqnarray}
\Courant{u,v,w}_H&:=&\Courant{u,v,w}_2+\Courant{H(u),v,w}_2+\Courant{u,H(v),w}_2+\Courant{u,v,H(w)}_2\label{LTS}\\
~ &&\nonumber+\Courant{H(u),H(v),w}_2+\Courant{u,H(v),H(w)}_2+\Courant{H(u),v,H(w)}_2,
\end{eqnarray}
defines a Lie triple system structure on $\g_2$.
\end{pro}
\begin{proof}
Since $((\huaG,\Theta^H),\g_1,\g_2)$ is a twilled Lie triple system, by Proposition \ref{pro:twilled}, we deduce that $\hat\mu_2^H$ is a Lie triple system structure on $\huaG$. Moreover, for all $u,v,w\in \g_2$, by Eqs. \eqref{control} and \eqref{equi}, we compute that
\begin{eqnarray*}
\hat\mu_2(u,v,w)&=&\Courant{u,v,w}_2,\\
~X_{\hat H}(\hat\psi)(u,v,w)&=&[\hat\psi,\hat H]_{\mathsf{LTS}}(u,v,w)=\Big(\hat\psi\circ\hat H-\hat H\circ\hat\psi\Big)(u,v,w)\\
~ &=&\hat\psi(H(u),v,w)+\hat\psi(u,H(v),w)+\hat\psi(u,v,H(w))\\
~ &=&\Courant{H(u),v,w}_2+\Courant{u,H(v),w}_2+\Courant{u,v,H(w)}_2,\\
~\half X_{\hat H}^2(\hat\mu_1)(u,v,w)&=&\half[[\hat\mu_1,\hat H]_{\mathsf{LTS}},\hat H]_{\mathsf{LTS}}(u,v,w)\\
~ &=&\Courant{H(u),H(v),w}_2+\Courant{H(u),v,H(w)}_2+\Courant{u,H(v),H(w)}_2.
\end{eqnarray*}
Adding these terms yields Eq. \eqref{LTS}. This completes the proof.
\end{proof}

\section{Matched pairs and relative Rota-Baxter operators}
In this section, we examine the relationship between twilled Lie triple systems and matched pairs and relationship between twilled Lie triple systems and relative Rota-Baxter operators respectively in order to obtain the relationship between matched pairs of Lie triple systems and relative Rota-Baxter operators. This work help us achieve bialgebra theory for Lie triple systems.
\subsection{Matched pairs and twilled Lie triple systems}
The notion of matched pairs of Lie-Yamaguti algebras was introduced in \cite{ZQ}, and here we restrict the Lie-Yamaguti algebras to the context of Lie triple systems to give the definition of matched pairs of Lie triple systems.
\begin{defi}
A {\bf matched pair} of Lie triple systems consists of two Lie triple systems $(\g_1,\Courant{\cdot,\cdot,\cdot}_{\g_1})$ and $(\g_2,\Courant{\cdot,\cdot,\cdot}_{\g_2})$, and two linear maps $\rho_1:\ot^2\g_1\longrightarrow\gl(\g_2)$ and $\rho_2:\ot^2\g_2\longrightarrow\gl(\g_1)$, such that the following conditions are satisfied:
\begin{itemize}
\item[\rm(i)] $(\g_2;\rho_1)$ is a representation of $\g_1$;
\item[\rm (ii)] $(\g_1;\rho_2)$ is a representation of $\g_2$;
\item[\rm (iii)] For all $x,y,z\in \g_1$ and $u,v,w\in \g_2$, the following equalities hold:
\begin{eqnarray*}
~ \rho_2(u,v)\Courant{x,y,z}_{\g_1}&=&\Courant{x,y,\rho_2(u,v)z}_{\g_1}-\rho_2(D_1(x,y)u,v)z-\rho_2(u,D_1(x,y)v)z,\\
~\label{mtp8} \Courant{x,y,\rho_2(u,v)z}_{\g_1}&=&\rho_2(u,D_1(x,y)v)z+\rho_2(\rho_1(z,y)u,v)x-\rho_2(\rho_1(z,x)u,v)y,\\
~\label{mtp10} \Courant{\rho_2(u,v)x,y,z}_{\g_1}&=&\rho_2(u,\rho_1(y,z)v)x+D_2(v,\rho_1(x,y)u)z-\rho_2(v,\rho_1(x,z)u)y,\\
~ \rho_1(x,y)\Courant{u,v,w}_{\g_2}&=&\Courant{u,v,\rho_1(x,y)w}_{\g_2}-\rho_1(D_2(u,v)x,y)w-\rho_1(x,D_2(u,v)y)w\\
~\label{mtp16}  \Courant{u,v,\rho_1(x,y)w}_{\g_2}&=&\rho_1(x,D_2(u,v)y)w+\rho_1(\rho_2(w,v)x,y)u-\rho_1(\rho_2(w,u)x,y)v,\\
~\label{mtp18} \Courant{\rho_1(x,y)u,v,w}_{\g_2}&=&\rho_1(x,\rho_2(v,w)y)u+D_1(y,\rho_2(u,v)x)y-\rho_1(y,\rho_2(u,w)x)v,
\end{eqnarray*}
\end{itemize}
where $D_1(x,y)=D_{\rho_1}(x,y)=\rho_1(y,x)-\rho_1(x,y)$ and $D_2(u,v)=D_{\rho_2}(u,v)=\rho_2(v,u)-\rho_2(u,v)$. We denote a matched pair of Lie triple systems by a quadruple $(\g_1,\g_2;\rho_1,\rho_2)$.
\end{defi}

The following proposition demonstrates that a matched pair of Lie triple systems gives rise to a Lie triple systems structure on its direct sum.

\begin{pro}\label{pro:MT}
Let $(\g_1,\g_2;\rho_1,\rho_2)$ be a matched pair of Lie triple systems, where $\rho_1:\ot^2\g_1\longrightarrow\gl(\g_2)$ and $\rho_2:\ot^2\g_2\longrightarrow\gl(\g_1)$. Define an operation $\Courant{\cdot,\cdot,\cdot}_{\bowtie}$ on $\g_1\oplus\g_2$ to be
\begin{eqnarray}
\Courant{x+u,y+v,z+w}_{\bowtie}&=&\Courant{x,y,z}_{\g_1}+D_2(u,v)z+\rho_2(v,w)x-\rho_2(u,w)y\label{double}\\
~ \nonumber&&+\Courant{u,v,w}_{\g_2}+D_1(x,y)w+\rho_1(y,z)u-\rho_1(x,z)v.
\end{eqnarray}
Then $(\g_1\oplus\g_2,\Courant{\cdot,\cdot,\cdot}_{\bowtie})$ is a Lie triple system, which is called the {\bf double} of $\g_1$ and $\g_2$.
\end{pro}
\begin{proof}
Let the quadruple $(\g_1,\g_2;\rho_1,\rho_2)$ be a matched pair of Lie triple systems, then we have the following equations:
\begin{eqnarray*}
\left\{\begin{array}{rcl}
[\hat\pi_1,\hat\rho_1]_{\mathsf{LTS}}+\half [\hat\rho_1,\hat\rho_1]_{\mathsf{LTS}}&=&0,\\
~[\hat\pi_2,\hat\rho_2]_{\mathsf{LTS}}+\half [\hat\rho_2,\hat\rho_2]_{\mathsf{LTS}}&=&0,\\
~[\hat\pi_2,\hat\rho_1]_{\mathsf{LTS}}+[\hat\rho_2,\hat\pi_1+\hat\rho_1]_{\mathsf{LTS}}&=&0.
\end{array}\right.
\end{eqnarray*}
Here, $\hat\pi_1,~\hat\rho_1,~\hat\pi_2,~\hat\rho_2\in C^1_{\mathsf{LTS}}(\huaG,\huaG)$ are given by
\begin{eqnarray*}
\hat\pi_1((x,u),(y,v),(z,w))&=&\Courant{x,y,z}_{\g_1},\\
\hat\rho_1((x,u),(y,v),(z,w))&=&D_1(x,y)w+\rho_1(y,z)u-\rho_1(x,z)v,\\
~\hat\pi_2((x,u),(y,v),(z,w))&=&\Courant{u,v,w}_{\g_2},\\
 \hat\rho_2((x,u),(y,v),(z,w))&=&D_2(u,v)z+\rho_2(v,w)x-\rho_2(u,w)y,\quad \forall x,y,z\in \g_1,~u,v,w\in \g_2,
\end{eqnarray*}
respectively.
It is straightforward to see that the equation
$$[\hat\pi_1,\hat\rho_1]_{\mathsf{LTS}}+\half [\hat\rho_1,\hat\rho_1]_{\mathsf{LTS}}=0 ~~(\text{resp.}~~[\hat\pi_2,\hat\rho_2]_{\mathsf{LTS}}+\half [\hat\rho_2,\hat\rho_2]_{\mathsf{LTS}}=0)$$
is equivalent to $(\g_2;\rho_1)$ (resp. $(\g_1;\rho_2)$) is a representation of $\g_1$ (resp. $\g_2$), and that the equation $$[\hat\pi_1,\hat\rho_2]_{\mathsf{LTS}}+[\hat\rho_1,\hat\pi_2+\hat\rho_2]_{\mathsf{LTS}}=0$$
is equivalent to Condition (iii).
Since $\g_1$ and $\g_2$ are Lie triple systems, i.e.,
$$[\pi_1,\pi_1]_{\mathsf{LTS}}=[\pi_2,\pi_2]_{\mathsf{LTS}}=0,$$
then we have
$$[\hat\pi_1+\hat\rho_1+\hat\pi_2+\hat\rho_2,\hat\pi_1+\hat\rho_1+\hat\pi_2+\hat\rho_2]_{\mathsf{LTS}}=0,$$
which implies that $(\g_1\oplus\g_2,\Courant{\cdot,\cdot,\cdot}_{\bowtie})$ is a Lie triple system. This finishes the proof.
\end{proof}

To demonstrate the relationship between matched pairs and twilled  Lie triple systems, we introduce the notion of strict twilled  Lie tripe systems as follows.
\begin{defi}
Let $(\huaG,\g_1,\g_2)$ be a twilled Lie triple system equipped with the structure $\Theta=\hat\mu_1+\hat\psi+\hat\mu_2$. Then $(\huaG,\g_1,\g_2)$ is called {\bf strict} if $\psi=0$, or equivalently, the following conditions are satisfied:
\begin{eqnarray*}
\left\{\begin{array}{rcl}
~\half[\hat\mu_1,\hat\mu_1]_{\mathsf{LTS}}&=&0,\\
~[\hat\mu_1,\hat\mu_2]_{\mathsf{LTS}}&=&0,\\
~\half[\hat\mu_2,\hat\mu_2]_{\mathsf{LTS}}&=&0.
\end{array}\right.
\end{eqnarray*}
\end{defi}

\begin{thm}\label{strict}
There is a one-to-one correspondence between matched pairs and strict twilled Lie triple systems.
\end{thm}
\begin{proof}
Let the quadruple $(\g_1,\g_2;\rho_1,\rho_2)$ be a matched pair of Lie triple systems, then by Proposition \ref{pro:MT}, the double operation $\Courant{\cdot,\cdot,\cdot}_{\bowtie}$ is a Lie triple system structure on the direct sum of vector space $\g_1\oplus\g_2$. The letter is equivalent to that $\hat\mu_1+\hat\mu_2\in C^1_{\mathsf{LTS}}(\huaG,\huaG)$ is a Maurer-Cartan element of the graded Lie algebra $(C^*_{\mathsf{LTS}}(\huaG,\huaG),[\cdot,\cdot]_{\mathsf{LTS}})$, i.e.,
$$[\hat\mu_1+\hat\mu_2,\hat\mu_1+\hat\mu_2]_{\mathsf{LTS}}=0,$$
where $\hat\mu_1$ and $\hat\mu_2$ are given by
\begin{eqnarray*}
\hat\mu_1(x+u,y+v,z+w)&=&\Big(\Courant{x,y,z}_{\g_1},D_1(x,y)w+\rho_1(y,z)u-\rho_1(x,z)v\Big),\\
\hat\mu_2(x+u,y+v,z+w)&=&\Big(D_2(u,v)z+\rho_2(v,w)x-\rho_2(u,w)y,\Courant{u,v,w}_{\g_2}\Big),
\end{eqnarray*}
for all $x,y,z\in \g_1,~u,v,w\in \g_2$. By the convention in \eqref{equi}, define
$$\rho_1(x,y)u:=\Courant{u,x,y}_2, \quad \rho_2(u,v)x:=\Courant{x,u,v}_1.$$
Consequently we obtain
$$D_1(x,y)u=\rho_1(y,x)u-\rho_1(x,y)u=\Courant{x,y,u}_2,$$
and similarly
$$D_2(u,v)x=\rho_2(v,u)x-\rho_2(u,v)x=\Courant{u,v,x}_1.$$
Then the structure $\Theta=\hat\mu_1+\hat\mu_2$, which demonstrates that $\huaG=\g_1\oplus\g_2$ is a strict twilled Lie triple system. This finishes the proof.
\end{proof}

It is natural to ask what object a usual twilled Lie triple system corresponds to. To answer this question, we shall introduce the notion of generalized matched pairs of Lie triple systems. Before this, we give the definition of generalized representation of Lie triple systems first.

For a linear map $\tau:\g\longrightarrow\Hom(\ot^2V,V)$, there induces a linear map $\huaD_{\tau}:\g\longrightarrow\Hom(\wedge^2V,V)$ defined to be
\begin{eqnarray}
\huaD_\tau(x)(u,v)=\tau(x)(v,u)-\tau(x)(u,v),\quad \forall x\in \g,~u,v\in V.\label{huaD}
\end{eqnarray}
We denote $\huaD_\tau$ by $\huaD$ without ambiguities.
\begin{defi}
Let $(\g,\Courant{\cdot,\cdot,\cdot})$ be a Lie triple system and $V$ a vector space. A {\bf generalized representation} of $\g$ on $V$ consists of linear maps $\rho:\ot^2\g\longrightarrow\gl(V)$ and $\tau:\g\longrightarrow\Hom(\ot^2V,V)$, such that $\hat\pi+\hat\rho+\hat\tau\in C^1_{\mathsf{LTS}}(\g\oplus V,\g\oplus V)$ is a Maurer-Cartan element of $(C^*_{\mathsf{LTS}}(\g\oplus V,\g\oplus V),[\cdot,\cdot]_{\mathsf{LTS}})$, i.e.,
$$[\hat\pi+\hat\rho+\hat\tau,\hat\pi+\hat\rho+\hat\tau]_{\mathsf{LTS}}=0.$$
Here $\hat\pi,~\hat\rho,~\hat\tau$ are given by
\begin{eqnarray*}
\hat\pi(x+u,y+v,z+w)&=&\Courant{x,y,z},\\
 \hat\rho(x+u,y+v,z+w)&=&D(x,y)w+\rho(y,z)u-\rho(x,z)v, \\
\hat\tau(x+u,y+v,z+w)&=&\huaD(z)(u,v)+\tau(x)(v,w)-\tau(y)(u,w),\quad \forall x,y,z\in \g,~~u,v,w\in V,
\end{eqnarray*}
respectively, where $\huaD$ is defined by \eqref{huaD}. We denote a generalized representation of a Lie triple system $\g$ by a pair $(V;(\rho,\tau))$.
\end{defi}

\begin{ex}
Let $(\g,\Courant{\cdot,\cdot,\cdot})$ be a Lie triple system. Define
$\huaR:\g\longrightarrow\Hom(\ot^2\g,\g)$
to be
$$\huaR(x)(y,z):=\Courant{x,y,z}, \forall x,y,z\in \g.$$
Then $(\g;(\frkR,\huaR))$ is a generalized representation of $\g$ on itself, where $(\g;\frkR)$ is the adjoint representation of $\g$. Obviously, $\huaD_\huaR=\huaL:\g\longrightarrow\Hom(\wedge^2\g,\g)$ is given by
$$\huaL(x)(y,z)=\Courant{y,z,x},\quad \forall x,y,z\in \g.$$
\end{ex}

Generalized representation can be characterized by the notion of generalized semidirect product Lie triple systems.
\begin{pro}
Let $(\g,\Courant{\cdot,\cdot,\cdot})$ be a Lie triple system, $V$ a vector space, and $\rho:\ot^2\g\longrightarrow\gl(V)$ and $\tau:\g\longrightarrow\Hom(\ot^2V,V)$ be linear maps. Then $(V;(\rho,\tau))$ is a generalized representation of $\g$ if and only if there exists a Lie triple system structure $\Courant{\cdot,\cdot,\cdot}_{(\rho,\tau)}$ on the direct sum $\g\oplus V$ defined to be for all $x,y,z\in \g$ and $u,v,w\in V$,
\begin{eqnarray*}
\Courant{x+u,y+v,z+w}_{(\rho,\tau)}&=&\Courant{x,y,z}+D(x,y)w+\rho(y,z)u-\rho(x,z)v\\
~ &&+\huaD(z)(u,v)+\tau(x)(v,w)-\tau(y)(u,w).
\end{eqnarray*}
The Lie triple system $(\g,\Courant{\cdot,\cdot,\cdot}_{(\rho,\tau)})$ is called the {\bf generalized product Lie triple system}.
\end{pro}
\begin{proof}
It follows from
$$\Courant{x+u,y+v,z+w}_{(\rho,\tau)}=(\hat\pi+\hat\rho+\hat\tau)(x+u,y+v,z+w)$$
and involves a direct computation.
\end{proof}

\begin{rmk}
The notion of generalized representation of Lie triple systems given, it is natural to establish a new cohomology for Lie triple systems and then  deformations and extensions can be explored consequently. We will examine this problem in the future and we are also looking forward to new studies in this direction.
\end{rmk}

Next, we introduce the notion of generalized matched pairs of Lie triple systems in the sequel.

\begin{defi}
A {\bf generalized matched pair} of Lie triple systems consists of two Lie triple systems $(\g_1,\Courant{\cdot,\cdot,\cdot}_{\g_1})$ and $(\g_2,\Courant{\cdot,\cdot,\cdot}_{\g_2})$, and two pairs of linear maps $\rho_1:\ot^2\g_1\longrightarrow\gl(\g_2),~\tau_1:\g_1\longrightarrow\Hom(\ot^2\g_2,\g_2)$ and $\rho_2:\ot^2\g_2\longrightarrow\gl(\g_1),~\tau_2:\g_2\longrightarrow\Hom(\ot^2\g_1,\g_1)$, such that the following equalities hold:
\begin{eqnarray}
[\hat\pi_1,\hat\rho_1+\hat\tau_1]_{\mathsf{LTS}}+\half [\hat\rho_1+\hat\tau_1,\hat\rho_1+\hat\tau_1]_{\mathsf{LTS}}&=&0,\label{rep1}\\
~[\hat\pi_2,\hat\rho_2+\hat\tau_2]_{\mathsf{LTS}}+\half [\hat\rho_2+\hat\tau_2,\hat\rho_2+\hat\tau_2]_{\mathsf{LTS}}&=&0,\label{rep2}\\
~[\hat\pi_1,\hat\rho_2+\hat\tau_2]_{\mathsf{LTS}}+[\hat\rho_1+\hat\tau_1,\hat\pi_2]_{\mathsf{LTS}}+[\hat\rho_1+\hat\tau_1,\hat\rho_2+\hat\tau_2]_{\mathsf{LTS}}&=&0,\label{matp}
\end{eqnarray}
where
\begin{eqnarray*}
\hat\pi_1((x,u),(y,v),(z,w))&=&\Courant{x,y,z}_{\g_1},\quad  \hat\rho_1((x,u),(y,v),(z,w))=D_1(x,y)w+\rho_1(y,z)u-\rho_1(x,z)v,\\
\hat\tau_1((x,u),(y,v),(z,w))&=&\huaD_1(z)(u,v)+\tau_1(x)(v,w)-\tau_1(y)(u,w),\\
\hat\pi_2((x,u),(y,v),(z,w))&=&\Courant{u,v,w}_{\g_2},\quad \hat\rho_2((x,u),(y,v),(z,w))=D_2(u,v)z+\rho_2(v,w)x-\rho_2(u,w)y,\\
\hat\tau_2((x,u),(y,v),(z,w))&=&\huaD_2(w)(x,y)+\tau_2(u)(y,z)-\tau(v)(x,z),\quad \forall x,y,z\in \g_1,~u,v,w\in \g_2.
\end{eqnarray*}
We denote a generalized matched pair of Lie triple systems by a quadruple $\Big(\g_1,\g_2;(\rho_1,\tau_1),(\rho_2,\tau_2)\Big)$ .
\end{defi}

If the quadruple $\Big(\g_1,\g_2;(\rho_1,\tau_1),(\rho_2,\tau_2)\Big)$ is a generalized matched pair of Lie triple systems, then it is straightforward to see that Eq. \eqref{rep1} (resp. Eq. \eqref{rep2}) means that $(\g_2;(\rho_1,\tau_1))$ (resp. $(\g_1;(\rho_2,\tau_2))$) is a generalized representation of $\g_2$ (resp. $\g_1$), and Eq. \eqref{matp} means certain compatibility conditions.

\begin{pro}
Suppose that the quadruple $\Big(\g_1,\g_2;(\rho_1,\tau_1),(\rho_2,\tau_2)\Big)$ is a generalized matched pair of Lie triple systems. For all $x,y,z\in \g_1$ and $u,v,w\in \g_2$, define a new operation $\Courant{\cdot,\cdot,\cdot}_{\g_1\oplus\g_2}$ on the direct sum $\g_1\oplus\g_2$ to be
\begin{eqnarray}
~ \nonumber&&\Courant{x+u,y+v,z+w}_{\g_1\oplus\g_2}\\
~ \nonumber&=&\Courant{x,y,z}_{\g_1}+D_2(u,v)z+\rho_2(v,w)x-\rho_2(u,w)y+\huaD_2(w)(x,y)+\tau_2(u)(y,z)-\tau_2(v)(x,z)\\
~ &&+\Courant{u,v,w}_{\g_2}+D_1(x,y)w+\rho_1(y,z)u-\rho_1(x,z)v+\huaD_1(z)(u,v)+\tau_1(x)(v,w)-\tau_1(y)(u,w),\label{gdouble}
\end{eqnarray}
then $(\g_1\oplus\g_2,\Courant{\cdot,\cdot,\cdot}_{\g_1\oplus\g_2})$ is a Lie triple system.
\end{pro}
\begin{proof}
It follows from
$$\Courant{x+u,y+v,z+w}_{\g_1\oplus\g_2}=[\hat\pi_1+\hat\rho_1+\hat\tau_1+\hat\pi_2+\hat\rho_2+\hat\tau_2,\hat\pi_1+\hat\rho_1+\hat\tau_1+\hat\pi_2+\hat\rho_2+\hat\tau_2]_{\mathsf{LYS}}(x+u,y+v,z+w).$$
This completes the proof.
\end{proof}

\vspace{3mm}
Now, we are ready to give our another key conclusion in this section.
\begin{thm}\label{usual}
There is a one-to-one correspondence between generalized matched pairs and twilled Lie triple systems.
\end{thm}
\begin{proof}
The quadruple $(\g_1,\g_2;(\rho_1,\tau_1),(\rho_2,\tau_2))$ is a generalized matched pair of Lie triple systems, then the operation $\Courant{\cdot,\cdot,\cdot}_{\g_1\oplus\g_2}$ is a Lie triple system on the direct sum of vector space $\g_1\oplus\g_2$. The letter is equivalent to that $\hat\mu_1+\hat\psi+\hat\mu_2\in C^1_{\mathsf{LTS}}(\huaG,\huaG)$ is a Maurer-Cartan element of the graded Lie algebra $(C^*_{\mathsf{LTS}}(\huaG,\huaG),[\cdot,\cdot]_{\mathsf{LTS}})$, i.e.,
$$[\hat\mu_1+\hat\psi+\hat\mu_2,\hat\mu_1+\hat\psi+\hat\mu_2]_{\mathsf{LTS}}=0,$$
where $\hat\mu_1,\hat\psi,\hat\mu_2$ are given by
\begin{eqnarray*}
\hat\mu_1(x+u,y+v,z+w)&=&\Big(\Courant{x,y,z}_{\g_1},D_1(x,y)w+\rho_1(y,z)u-\rho_1(x,z)v\Big),\\
\hat\psi(x+u,y+v,z+w)&=&\Big(\huaD_2(w)(x,y)+\tau_2(u)(x,y)-\tau_2(v)(x,z),\huaD_1(z)(u,v)+\tau_1(x)(v,w)-\tau_1(y)(u,w)\Big),\\
\hat\mu_2(x+u,y+v,z+w)&=&\Big(D_2(u,v)z+\rho_2(v,w)x-\rho_2(u,w)y,\Courant{u,v,w}_{\g_2}\Big),
\end{eqnarray*}
for all $x,y,z\in \g_1,~u,v,w\in \g_2$. By the convention in \eqref{equi}, define
\begin{eqnarray*}
\rho_1(x,y)u:=\Courant{u,x,y}_2, \quad \rho_2(u,v)x:=\Courant{x,u,v}_1,\\
 \tau_1(x)(u,v):=\Courant{x,u,v}_2,\quad \tau_2(u)(x,y):=\Courant{u,x,y}_1.
 \end{eqnarray*}
Consequently we obtain
$$\huaD_1(x)(u,v)=\Courant{u,v,x}_2,\quad \huaD_2(u)(x,y)=\Courant{x,y,u}_1.$$
Then the structure $\Theta=\hat\mu_1+\hat\psi+\hat\mu_2$, which demonstrates that $\huaG=\g_1\oplus\g_2$ is a twilled Lie triple system. This finishes the proof.
\end{proof}

\begin{rmk}
If we constructed the dual representation of a generalized representation, then we could establish the generalized bialgebra of Lie triple systems and even, the generalized Lie-Yamaguti bialgebra theory. We will examine this projection in the future.
\end{rmk}

\subsection{Relative Rota-Baxter operators and twilled Lie triple systems}
In this subsection, we construct a twilled Lie triple system via a relative Rota-Baxter operator and give some examples to end up with this section.
\begin{defi}
Let $(\g,\Courant{\cdot,\cdot,\cdot})$ be a Lie triple system and $(V;\rho)$ a representation of $\g$. Then a linear map $T:V\longrightarrow\g$ is called a {\bf relative Rota-Baxter operator} on $\g$ with respect to $(V;\rho)$ if $T$ satisfies
\begin{eqnarray*}
\Courant{Tu,Tv,Tw}=T\Big(D(Tu,Tv)w+\rho(Tv,Tw)u-\rho(Tu,Tw)v\Big),\quad \forall u,v,w\in V.
\end{eqnarray*}
Besides, a relative Rota-Baxter operator $T:\g\longrightarrow\g$ with respect to the adjoint representation $(\g;\frkR)$ is called a {\bf Rota-Baxter operator} on $\g$, i.e., $T$ satisfies
$$\Courant{Tx,Ty,Tz}=T\Big(\Courant{Tx,Ty,z}+\Courant{x,Ty,Tz}+\Courant{Tx,y,Tz}\Big),\quad \forall x,y,z\in \g.$$
\end{defi}

In \cite{Makhlouf}, relative Rota-Baxter operators are realized as  Maurer-Cartan elements in a suitable $L_\infty$-algebra $(C^*(V,\g),l_3)$ (This $L_\infty$-algebra was called the Lie 3-algebra in the literature). When the (strict) twilled Lie triple system is just the semidirect product $((\g\ltimes V,\Theta=\hat\mu_1),\g,V)$, the $L_\infty$-algebra constructed in Theorem \ref{Linfty} is just the controlling algebra of relative Rota-Baxter operators. Thus based on Theorem \ref{MC}, we have the following result to relate relative Rota-Baxter operators with twilled Lie triple systems.
\begin{pro}\label{Otwisting}
Let $(\g,\Courant{\cdot,\cdot,\cdot})$ be a Lie triple system, $(V;\rho)$ a representation of $\g$, and $T:V\longrightarrow\g$ a linear map. Then $((\g\oplus V,\Theta^T),\g,V)$ is a twilled Lie triple system if and only if $T$ is a relative Rota-Baxter operator.
\end{pro}
\begin{proof}
Since the semidirect product $((\g\ltimes V,\Theta),\g,V)$ is a strict Lie triple system, by Theorem \ref{twilledtwisting}, the twisting $\Theta^T$ has the structures as follows:
\begin{eqnarray*}
\left\{\begin{array}{rcl}
\hat\mu_1^T&=&\hat\mu_1,\\
\hat\psi^T&=&X_{\hat T}(\hat\mu_1),\\
\hat\mu_2^T&=&\half X_{\hat T}^2(\hat\mu_1),\\
\hat\phi_2^T&=&\frac{1}{6}X_{\hat T}^3(\hat\mu_1).
\end{array}\right.
\end{eqnarray*}
Thus $((\g\oplus V,\Theta^T),\g,V)$ is a twilled Lie triple system if and only if $\hat\phi_2^T=0$, which is equivalent to that $T$ is a relative Rota-Baxter operator.
\end{proof}

\begin{cor}
Let $T:V\longrightarrow\g$ be a relative Rota-Baxter operator on a Lie triple system $(\g,\Courant{\cdot,\cdot,\cdot})$ with respect to a representation $(V;\rho)$. Then $\hat\mu_2^T$ defines a Lie triple system structure $\Courant{\cdot,\cdot,\cdot}_{T}$ on $V$ and a representation $\varrho:\ot^2V\longrightarrow\gl(\g)$ as follows:
\begin{eqnarray*}
\Courant{u,v,w}_T&=&D(Tu,Tv)w+\rho(Tv,Tw)u-\rho(Tu,Tw)v,\\
\varrho(u,v)x&=&\Courant{x,Tu,Tv}-T\Big(D(x,Tu)v-\rho(x,Tv)u\Big), \quad \forall x\in \g,~u,v,w\in V.
\end{eqnarray*}
\end{cor}
\begin{proof}
By Proposition \ref{Otwisting}, $((\g\oplus V,\Theta^T),\g,V)$ is a twilled Lie triple system, which implies that $\hat\mu_2^T$ defines a Lie triple system structure on $V$. Moreover, for all $x\in \g$ and $u,v,w\in V$, define
$$\Courant{u,v,w}_T:=\hat\mu_2^T(u,v,w), \quad \varrho(u,v)x:=\hat\mu_2^T(u,v,x),$$
which give a Lie triple system structure on $ V$ and a representation of $(V,\Courant{\cdot,\cdot,\cdot}_{T})$ on $\g$.
\end{proof}

At the end of this paper, we need a proposition to give some examples. Before this, we introduce some notations. Let $(\g,\Courant{\cdot,\cdot,\cdot})$ be a Lie triple system, $(V;\rho)$ a representation and $T:V\longrightarrow\g$ a relative Rota-Baxter operator with respect to $(V;\rho)$. Let $\Theta$ denote the operation of twilled Lie triple system $(\g\ltimes V,\g,V)$. For all $x,y\in \g$ and $u,v\in V$, define $\tau:\g\longrightarrow\Hom(\ot^2V,V)$ and $\sigma:V\longrightarrow\Hom(\ot^2\g,\g)$ to be
\begin{eqnarray*}
\tau(x)(u,v)&=&\rho(Tv,x)u+\rho(x,Tu)v,\\
\sigma(u)(x,y)&=&\Courant{Tu,x,y}-T\Big(\rho(x,y)u\Big)
\end{eqnarray*}
respectively. Then we have the explicit formula of linear maps $\huaD_\tau:\g\longrightarrow\Hom(\wedge^2V,V)$ and $\huaD_\sigma:V\longrightarrow\Hom(\wedge^2\g,\g)$ as follows:
\begin{eqnarray*}
\huaD_\tau(x)(u,v)&=&D(x,Tu)v+D(Tv,x)u,\\
\huaD_\sigma(u)(x,y)&=&\Courant{x,y,Tu}-T\Big(D(x,y)u\Big),\quad \forall x,y\in \g,~u,v\in V.
\end{eqnarray*}
\begin{pro}
With the above notations, the twisting of $\Theta$ is given by
\begin{eqnarray}
~ &&\Theta^T\Big((x,u),(y,v),(z,w)\Big)\label{Ttwisting}\\
~ \nonumber&=&\Courant{x,y,z}+D_\varrho(u,v)z+\varrho(v,w)x-\varrho(u,w)y+\huaD_\sigma(w)(x,y)+\sigma(u)(y,z)-\sigma(v)(x,z)\\
~\nonumber &&+\Courant{u,v,w}_T+D_\rho(x,y)w+\rho(y,z)u-\rho(x,z)v+\huaD_\tau(z)(u,v)+\tau(x)(v,w)-\tau(y)(u,w).
\end{eqnarray}
\end{pro}

Note that although $((\g\oplus V,\Theta^T),\g,V)$ is a twilled Lie triple system, however nether is $(V;(\rho,\tau))$ a generalized representation of $\g$, nor $(\g;(\varrho,\sigma))$ is a generalized representation of $(V,\Courant{\cdot,\cdot,\cdot}_T)$ in general. Consequently, the quadruple $(\g,V;(\rho,\tau),(\varrho,\sigma))$ does not form a generalized matched pair any more.

\begin{ex}
Let $(\g,\Courant{\cdot,\cdot,\cdot})$ be a 2-dimension Lie triple system with a basis $\{e_1,e_2\}$. Define a nonzero operation with respect to the basis $\{e_1,e_2\}$ to be
$$\Courant{e_1,e_2,e_2}=e_1.$$
Then
\begin{eqnarray*}
T=
\begin{pmatrix}
0 & a\\
0 & b
\end{pmatrix}
\end{eqnarray*}
is a Rota-Baxter operator on $\g$. Then $((\g\oplus\g_T,\Theta^T),\g,\g_T)$ is a 4-dimensional Lie triple system with a basis $\{\bm{e}_1,\bm{e}_2,\bm{e}_3,\bm{e}_4\}$, where $\bm{e}_1=(e_1,0),~\bm{e}_2=(e_2,0),\bm{e}_3=(0,e_1),~\bm{e}_4=(0,e_2)$. By Eq. \eqref{Ttwisting}, the twisting $\Theta^T$ is given by
\begin{eqnarray*}
\Theta^T({\bm e}_1,{\bm e}_3,{\bm e}_2)=-{\bm e}_3,\quad \Theta^T(\bm{e}_2,\bm{e}_3,\bm{e}_4)=b\bm{e}_3,\quad \Theta^T(\bm{e}_2,\bm{e}_4,\bm{e}_3)=b\bm{e}_3,\\
\Theta^T(\bm{e}_1,\bm{e}_2,\bm{e}_4)=b\bm{e}_1+\bm{e}_3,\qquad \Theta^T(\bm{e}_1,\bm{e}_4,\bm{e}_2)=b\bm{e}_1+\bm{e}_3.
\end{eqnarray*}
\end{ex}

\begin{ex}
Let $(\g,\Courant{\cdot,\cdot,\cdot})$ be a 4-dimension Lie triple system with a basis $\{e_1,e_2,e_3,e_4\}$. Define a nonzero operation with respect to the basis $\{e_1,e_2,e_3,e_4\}$ to be
$$\Courant{e_1,e_2,e_1}=e_4.$$
Take
\begin{eqnarray*}
T=
\begin{pmatrix}
0 & 1 & 0 & 0\\
0 & 0 & 0 & 0\\
0 & 0 & 1 & 0\\
0 & 0 & 0 & 1
\end{pmatrix},
\end{eqnarray*}
then $T:\g\longrightarrow\g$ is a Rota-Baxter operator. Consequently, the twisting of Lie triple system $(\g\oplus\g_T,\g,\g_T)$ has a basis $\{\bm{ e}_1,\bm{e}_2\cdots,\bm{e}_8\}$, where $\bm{e}_i=(e_i,0), 1\leqslant i\leqslant 4$ and $\bm{e}_j=(0,e_{j-4}), 5\leqslant j\leqslant 8$. Then the twisting $\Theta^T$ is given by
\begin{eqnarray*}
\Theta^T(\bm{e}_1,\bm{e}_2,\bm{e}_1)&=&\bm{e}_4, \quad \Theta^T(\bm{e}_1,\bm{e}_2,\bm{e}_6)=\bm{e}_4, \quad \Theta^T(\bm{e}_2,\bm{e}_6,\bm{e}_1)=\bm{e}_4,\quad \Theta^T(\bm{e}_1,\bm{e}_2,\bm{e}_5)=\bm{e}_6,\\
\Theta^T(\bm{e}_2,\bm{e}_5,\bm{e}_1)&=&-\bm{e}_6, \quad \Theta^T(\bm{e}_2,\bm{e}_6,\bm{e}_5)=\bm{e}_6,\quad \Theta^T(\bm{e}_5,\bm{e}_6,\bm{e}_2)=\bm{e}_6,\quad \Theta^T(\bm{e}_1,\bm{e}_6,\bm{e}_1)=\bm{e}_4+\bm{e}_6.
\end{eqnarray*}
\end{ex}
\emptycomment{
\vspace{3mm}
\noindent{\bf Conclusion.}
In \cite{ZQ}, where Lie-Yamaguti bialgebra theory was established, it was found that a classical Lie-Yamaguti $r$-matrix does {\em not} give rise to a double Lie-Yamaguti bialgebra structure. Also, the same phenomenon also happens when Lie-Yamaguti  algebra restricted to the context of Lie triple systems. By Theorem \ref{strict}, Theorem \ref{usual} and Proposition \ref{Otwisting}, in contrast to binary cases such as Lie algebras or Leibniz algebras, we see that a relative Rota-Baxter operator does {\em not} correspondence to a usual matched pairs of Lie triple systems. These results explains why the classical $r$-matrix (treated as a relative Rota-Baxter operators with respect to the coadjoint representation) of Lie triple systems does not give rise to a double bialgebra of Lie triple system. Anyway, we use the following diagrams to indicate the relations among relative Rota-Baxter operators, matched pairs and twilled Lie triple systems.}

\end{document}